\documentclass[a4paper,11pt,reqno]{amsart}
\usepackage{mathabx}
\usepackage{lmodern}
\usepackage[utf8]{inputenc}
\usepackage[left=2cm,right=2cm,bottom=3cm,top=3cm]{geometry}
\usepackage[all,cmtip]{xy}
\usepackage{amsmath}
\usepackage{amsthm}
\usepackage{subcaption}
\usepackage{bbm} 
\usepackage{bm} 

\usepackage{amssymb}
\usepackage[utf8]{inputenc}
\usepackage[T1]{fontenc}

\usepackage{mathrsfs}
\usepackage{enumerate}
\usepackage{tikz-cd}
\usepackage{lmodern}

\usepackage{relsize}
\newcommand\cca[1]{{%
  \ooalign{\raisebox{-.7ex}{\larger[5]$\circlearrowleft$}\cr
    \hidewidth$\,#1$\hidewidth}}}

\newtheorem{thm}{Theorem}[section]
\newtheorem{prop}[thm]{Proposition}
\newtheorem{cor}[thm]{Corollary}
\newtheorem{lem}[thm]{Lemma}
\theoremstyle{definition}
\newtheorem{defi}[thm]{Definition}
\newtheorem{claim}{Claim}[thm]
\newtheorem{rem}[thm]{Remark}
\newtheorem{ex}[thm]{Example}
\newtheorem{que}[thm]{Question}
\newtheorem{prob}[thm]{Problem}

\newcommand{\U}{\mathcal{U}}

\newcommand{\varep}{\varepsilon}

\newcommand{\A}{\mathcal A}
\newcommand{\G}{\mathcal G}
\newcommand{\M}{\mathcal M}

\newcommand{\Q}{\mathbb{Q}}
\newcommand{\R}{\mathbb{R}}
\newcommand{\GG}{\mathbb{G}}

\newcommand{\N}{\mathbb{N}}

\newcommand{\K}{\mathcal{K}}
\newcommand{\al}{\alpha}

\newcommand{\ga}{\gamma}
\newcommand{\de}{\delta}

\newcommand{\la}{\lambda}

\newcommand{\mc}[1]{\mathcal{#1}}

\newcommand{\mr}[1]{\mathrm{#1}}
\newcommand{\sig}{\sigma}

\def\acts{\curvearrowright}
\newcommand{\vertiii}[1]{{\left\vert\kern-0.25ex\left\vert\kern-0.25ex\left\vert #1 
    \right\vert\kern-0.25ex\right\vert\kern-0.25ex\right\vert}}

\DeclareMathOperator{\Age}{Age}
\DeclareMathOperator{\Flim}{Flim}
\DeclareMathOperator{\Iso}{Iso}
\DeclareMathOperator{\Emb}{Emb}
\DeclareMathOperator{\osc}{osc}
\DeclareMathOperator{\id}{id}

\newcommand{\nrm}[1]{\|#1\|}

\begin{document}

\title{Fra\"iss\'e and Ramsey properties of Fr\'echet spaces}
\date{\today}

\author[J. K. Kawach]{Jamal K. Kawach}
\thanks{}
\address{Department of Mathematics, University of Toronto, Toronto, Canada, M5S 2E4.}
\email{jamal.kawach@mail.utoronto.ca}
\urladdr{https://www.math.toronto.edu/jkawach}

\author[J. L\'opez-Abad]{Jordi L\'opez-Abad}
\thanks{}
\address{Departamento de Matem\'aticas Fundamentales, Facultad de Ciencias, UNED, 28040 Madrid, Spain}
\email{abad@mat.uned.es}
\urladdr{}

\thanks{\textcolor{black}{J.K.K. was partially supported by a Queen Elizabeth II Graduate Scholarship in Science and Technology. J.L.-A. was
partially supported by the Spanish Ministry of Science and Innovation grant PID2019-107701GB-I00 (Spain)}}

\begin{abstract}
We develop the theory of Fra\"iss\'e limits for classes of finite-dimensional multi-seminormed spaces, which are defined to be vector spaces equipped with a finite sequence of seminorms. We define a notion of a Fra\"iss\'e Fr\'echet space and we use the Fra\"iss\'e correspondence in this setting to obtain many examples of such spaces. This allows us to give a Fra\"iss\'e-theoretic construction of $(\GG^\omega, (\|\cdot\|_n)_{n<\omega})$, the separable Fr\'echet space of almost universal disposition for the class of all finite-dimensional Fr\'echet spaces with an infinite sequence of seminorms. We then identify and prove an approximate Ramsey property for various classes of finite-dimensional multi-seminormed spaces using known approximate Ramsey properties of normed spaces. A version of the Kechris-Pestov-Todor\v{c}evi\'c correspondence for approximately ultrahomogeneous Fr\'echet spaces is also established and is used to obtain new examples of extremely amenable groups. In particular, we show that the group of surjective linear seminorm-preserving isometries of $(\GG^\omega, (\|\cdot\|_n)_{n<\omega})$ is extremely amenable.
\end{abstract}

\maketitle

\section{Introduction}

The study of universal Banach spaces was initiated by Banach and Mazur \cite{Banach}, who showed that every separable Banach space embeds isometrically into $C([0,1])$ (see also \cite[Corollary 12.14]{Carothers} or \cite[Theorem 8.7.2]{Semadeni}). Motivated by the Rotation Problem of Mazur, Gurarij \cite{Gurarij} constructed a separable Banach space $\GG$, now known as the \emph{Gurarij space}, which is universal for separable Banach spaces and which satisfies the following additional property: For every $\varep > 0$, every pair $X \subseteq Y$ of finite-dimensional Banach spaces, and every isometric embedding $f : X \rightarrow \GG$, there is an $\varep$-isometric embedding $\tilde{f} : Y \rightarrow \GG$ which extends $f$. In other words, $\GG$ is \emph{of almost universal disposition} for the class of all finite-dimensional Banach spaces. Lusky \cite{Lusky} later showed that the Gurarij space is unique up to isometry, in the sense that any other space which is of almost universal disposition for all finite-dimensional Banach spaces must be isometrically isomorphic to $\GG$. (A simpler proof of the uniqueness and universality of the $\GG$ can be found in \cite{KS}.) As a consequence of the fact that $\GG$ is of almost universal disposition, one can show that $\GG$ is \emph{approximately ultrahomogeneous}: For every finite-dimensional subspace $X \subseteq \GG$, every $\varep > 0$ and every isometric embedding $f : X \rightarrow \GG$, there is a surjective linear isometry $g : \GG \rightarrow \GG$ such that $\|g(x) - f(x)\| \leq \varep$ for each $x \in X$. For more information on the Mazur Rotation Problem and related concepts in the context of Banach spaces, we refer the reader to the recent survey \cite{CFR}.

Similar developments have taken place in the setting of Fr\'echet spaces: Mazur and Orlicz (see, e.g., \cite[p. 101]{Rolewicz}) showed that $C(\R)$ is homeomorphically universal for the class of all separable Fr\'echet spaces, in the sense that every Fr\'echet space admits a linear homeomorphism into $C(\R)$. Moving toward the isometric theory of Fr\'echet spaces, Bargetz, K\k{a}kol and Kubi\'s \cite{BKK} recently proved the existence of an analogue of the Gurarij space in the setting of Fr\'echet spaces equipped with a fundamental system of seminorms; they construct a separable Fr\'echet space $(\GG^\omega, (\|\cdot\|_n)_{n < \omega})$ with a suitable sequence of seminorms which is approximately ultrahomogeneous (with respect to isometric embeddings which preserve each given seminorm simultaneously) for the class of all finite-dimensional Fr\'echet spaces equipped with an infinite sequence of seminorms. A similar space is also constructed in \cite{BKK} for the class of all finite-dimensional \emph{graded} Fr\'echet spaces, i.e. Fr\'echet spaces equipped with an increasing sequence of seminorms.

Interest in universal spaces such as the Gurarij space is partially motivated by the study of their groups of linear isometries. In particular, the isometry group $\Iso(\GG)$ of the Gurarij space has been an object of intense investigation in recent years (see, for instance, \cite{BLLM, BY1}). One way to develop a better understanding of Polish groups such as $\Iso(\GG)$ is via their topological dynamics; a relevant property here is that of \emph{extreme amenability}. Recall that a topological group $G$ is extremely amenable if every continuous action of $G$ on a compact Hausdorff space $X$ has a common fixed point, i.e. a point $x \in X$ such that $g \cdot x = x$ for all $g \in G$. Many examples of such groups have been obtained using a link --  known as the \emph{Kechris-Pestov-Todor\v{c}evi\'c (KPT) correspondence} -- between the Ramsey property and the extreme amenability of automorphism groups of ultrahomogeneous first-order structures. Examples of such structures come from \emph{Fra\"iss\'e theory}, which provides a correspondence between countable ultrahomogeneous structures and certain classes of finitely-generated first-order structures, known as \emph{Fra\"iss\'e classes}. Since Fra\"iss\'e's original paper \cite{Fraisse} (see also \cite{Hodges}), many other versions of the Fra\"iss\'e correspondence have been developed; one of the more prominent instances of this is the development the Fra\"iss\'e theory of \emph{metric structures} (i.e. metric spaces with additional compatible structure) due to Ben Yaacov \cite{BY}. Other presentations of Fra\"iss\'e theory in the context of structures from functional analysis, including ``non-commutative'' structures, have been developed recently in \cite{FLMT, Lupini2018}. Naturally, these events led to the establishment of the KPT correspondence in more general settings, including Banach spaces \cite{BLLM, FLMT}. In fact, the first successful application of the KPT correspondence for metric structures developed in \cite{MT} was recently achieved in \cite{BLLM}, where the authors show that $\Iso(\GG)$ is extremely amenable by proving an \emph{approximate Ramsey property} for the class of all finite-dimensional Banach spaces. Interestingly, the latter result makes crucial use of the Dual Ramsey Theorem of Graham and Rothschild \cite{GR}. In this paper we obtain analogous results in setting of Fr\'echet spaces and more generally for \emph{multi-seminormed} spaces, i.e. vector spaces equipped with a finite or infinite sequence of seminorms. In particular, we develop Fra\"iss\'e theory for multi-seminormed spaces together with the relevant versions of the KPT correspondence and the approximate Ramsey property. A similar approach was used in \cite{BLLM_2} to study different classes of exact operator spaces, where a sequence of norms is also an essential part of the structure. In the present paper we provide a Fra\"iss\'e-theoretic proof of the existence of the spaces constructed in \cite{BKK}, which were originally defined using properties of a universal (Fra\"iss\'e) operator on $\GG$ originally considered in \cite{GK}. Such an operator can be seen as a ``Gurarij'' version of the Rota universal operator on the separable infinite-dimensional Hilbert space (for which we refer the reader to, e.g., \cite[Section 4.1]{Lupini2018}). In addition to the known examples discussed above, we also obtain many new examples of Fr\'echet spaces with strong forms of approximate ultrahomogeneity, including examples which naturally arise from various combinations of Hilbert spaces and $L_p$ spaces. As a result, we obtain many new examples of extremely amenable groups; in particular, we show that the group $\Iso(\GG^\omega, (\|\cdot\|_n)_{n<\omega})$ of all surjective linear seminorm-preserving isometries of $(\GG^\omega, (\|\cdot\|_n)_{n<\omega})$ is extremely amenable.

The rest of the paper is organized as follows. In Section 2 we consider amalgamation and Fra\"iss\'e properties of finite-dimensional multi-seminormed spaces. In particular, we show that the class $\M_{<\omega}$ of all finite-dimensional vector spaces equipped with a finite sequence of seminorms is Fra\"iss\'e. Similarly, the classes consisting of all finite-dimensional vector spaces with a sequence of seminorms of length at most $n$ (for various $n\geq 1$) are shown to be Fra\"iss\'e as long as one restricts to \emph{separated} multi-seminormed spaces, where a space is separated if the natural topology induced by the associated sequence of seminorms is a Hausdorff topology. More generally, we show that Fra\"iss\'e classes of finite-dimensional normed spaces naturally give rise to Fra\"iss\'e classes of finite-dimensional multi-seminormed spaces, and we use this to obtain many examples of such classes. In Section 3 we develop a notion of a \emph{Fra\"iss\'e Fr\'echet space} and we obtain examples of such Fr\'echet spaces by taking \emph{Fra\"iss\'e limits} of classes of certain finite-dimensional multi-seminormed spaces. We also show that the separable Fr\'echet spaces of almost universal disposition constructed in \cite{BKK} can be realized as Fra\"iss\'e limits. Section 4 contains a proof of the analogue of the KPT correspondence for multi-seminormed spaces, which involves isolating a useful version of the approximate Ramsey property for such spaces. To conclude, we show how to transfer the approximate Ramsey property of classes of finite-dimensional normed spaces to the corresponding approximate Ramsey property for certain classes of multi-seminormed spaces, thus obtaining new examples of extremely amenable groups by way of the KPT correspondence.
\subsubsection*{Acknowledgments}

We are grateful to Félix Cabello Sánchez for several helpful conversations and remarks. We also thank the referee for many useful comments and corrections, as well as for suggesting Problem \ref{prob3} below.

\section{Fra\"iss\'e classes of finite-dimensional Fr\'echet spaces}

We use standard set-theoretic notation; in particular, $\omega = \{0, 1, \dots\}$ denotes the first countably infinite ordinal and $\N$ will denote $\omega \setminus \{0\}$. When convenient, an integer $n$ will be identified with the set $\{0,\dots, n-1\}$ of its predecessors.

Let $X$ be a Fr\'echet space, i.e. a locally convex Hausdorff topological vector space which is completely metrizable via a translation-invariant metric.  A \emph{fundamental system of seminorms} on $X$ is a sequence of seminorms $(\|\cdot\|_{X,n})_{n < \alpha}$ on $X$, where $\alpha \leq \omega$, such that the sets $$B_{n,\varep}(x) = \{y \in X : \max_{m < n} \|x-y\|_{X,m} < \varep\}, \, \, x \in X, n < \alpha, \varep > 0$$ form a basis for the topology of $X$. It is a standard fact that every Fr\'echet space admits a (non-unique) fundamental system of seminorms and so we can view a Fr\'echet space as a \emph{multi-seminormed space} of the above form, i.e. a (finite or infinite) tuple consisting of a topological vector space equipped with a sequence of seminorms. By definition, a fundamental system of seminorms is always \emph{separating}: For every non-zero $x \in X$ there is a seminorm $\|\cdot\|_{X,n}, n < \alpha$ such that $\|x\|_{X,n} \neq 0$; in this case we also say that $X$ is \emph{separated}, which happens precisely when the sequence of seminorms induces a Hausdorff topology on $X$. In this paper we will work more generally with multi-seminormed spaces of the form $\textbf{X} = (X, (\|\cdot\|_{X,n})_{n<\lambda_{\textbf{X}}})$, where $X$ is a vector space, $1 \leq \lambda_{\textbf{X}} \leq \omega$ is the \emph{length} of $\textbf{X}$, which is defined as the length of the associated sequence of seminorms (which we always assume is non-empty), and where $(\|\cdot\|_{X,n})_{n < \la_X}$ is a (finite or infinite) sequence of seminorms on $X$. We will always conflate the above tuple $\textbf{X}$ with its underlying space $X$, and so we will always write $\la_X$ for the length when the sequence of seminorms is understood. We point out that in general we do not require a multi-seminormed space to be separated. With this in mind, we reserve the term ``Fr\'echet space'' for a separated multi-seminormed space $(X, (\|\cdot\|_{X,n})_{n < \la_X})$ such that the topology generated by the seminorms $\|\cdot\|_{X,n}$ is complete. It is clear that in this case the sequence of seminorms $(\|\cdot\|_{X,n})$ forms a fundamental system of seminorms on $X$. Finally, following the terminology in \cite{BKK, vogt1987}, we say that a multi-seminormed space $(X, (\|\cdot\|_{X,n})_{n < \la_X})$ is \emph{graded} if, for all $x \in X$, $\|x\|_{X,n} \leq \|x\|_{X,m}$ whenever $n \leq m < \la_X$. For more information about (graded) Fr\'echet spaces, we refer the reader to \cite{Hamilton, MV, vogt1987}. The reader is also referred to \cite{Kutateladze, Simon} for more information about the general theory of multi-seminormed spaces (which are also sometimes called seminormed spaces or multinormed spaces in the literature).

Throughout this paper we will work with various subclasses of the class $\mc M$ of finite-dimensional multi-seminormed spaces equipped with a fundamental system of seminorms. Of particular interest will be the subclass $\G$ consisting of all graded $X \in \mc M$. Given any subclass $\K \subseteq \mc M$, and $\al \leq \omega$, we let $\K_{< \al}, \K_{\leq \al}$ and $\K_{=\al}$ denote the collections of all $X \in \K$ such that $\la_X < \al, \la_X \leq \al$ and $\la_X = \al$, respectively. For any such $\K$ we also let $\K^{\mathrm{sep}}$ denote the subclass of $\K$ consisting of all separated members of $\K$.

Given two multi-seminormed spaces $X$ and $Y$ and a linear mapping $f: X\to Y$, we say that a linear mapping $f:X\to Y$ is {\em multi-bounded} when $\la_X\le \la_Y$ and when for every $m<\la_X$ one has $$\nrm{f}_\mr{mb}:= \sup_{m<\la_X} \nrm{f}_{(X,\nrm{\cdot}_{X,m}), (Y,\nrm{\cdot}_{Y,m})}<\infty,$$ and where  as usual $ \nrm{f}_{(X,\nrm{\cdot}_{X,m}), (Y,\nrm{\cdot}_{Y,m})}$ is the (semi)norm of $f$ defined as $\sup_{\nrm{x}_{X,m}\le 1}\nrm{f(x)}_{Y,m}$. When working with the operator norm, we will usually suppress the notation and omit the reference to the underlying spaces when this is clear from context. Observe that since we are dealing with seminorms, this quantity might be infinite, even when $X$ is finite-dimensional.  A {\em multi-isomorphism} $f:X\to Y$ is linear bijection such that both $f$ and $f^{-1}$ are multi-bounded.  Observe also that an isomorphism is a Lipschitz mapping, when considering the metrics associated to the fixed seminorms on $X$ and $Y$.  

Much like in the class of normed spaces and others, we define a Banach-Mazur-like pseudometric, quantifying distances between isometric types. (The relevant notion of an isometry in this setting will be defined below.) We define
$$d_\mr{BM}(X,Y):=\log\inf_{f}\nrm{f}_\mr{mb} \nrm{f^{-1}}_\mr{mb},$$
where $f$ runs over all multi-isomorphisms (if any) between $X$ and $Y$. 

\begin{prop}
${\mc M}_{<\omega}$ and ${\mc G}_{<\omega}$ are $\sig$-compact, and consequently separable. 
\end{prop}  
\begin{proof}
${\mc G}_{<\omega}$ is closed in ${\mc M}_{<\omega}$, so it suffices to prove that ${\mc M}_{<\omega}$ is $\sig$-compact. As in the case of the Banach-Mazur pseudometric on the class of finite-dimensional normed spaces, it is not difficult to see that the closed $d_\mr{BM}$-balls on ${\mc M}_{<\omega}$ are compact. But unlike for the normed space case, where for each dimension $k$ the diameter is finite (hence the Banach-Mazur compactum is compact), this is not the case for multi-seminormed spaces.  Given $(X,(\nrm{\cdot}_{X,k})_{k<\la_X})\in {\mc M}_{<\omega}$, let $\overline{\alpha}_X=(\alpha^X_s)_{s\subseteq \la_X}$ be defined for $s\subseteq \la_X$ by $\al_s^X:=  \dim \bigcap_{k\in s} \ker \nrm{\cdot}_{X,k}$ and $\al_\varnothing^X := \dim X$.  
\begin{claim}
 The classes of multi-seminormed spaces with finite distances (i.e. isomorphic classes) are exactly,  given a sequence $\bar\al=(\al_s)_{s\subseteq n}$ of positive integers,  the sets $\mc A_{\bar{\al}}:=\{X\in {\mc M}_{<\omega} \,:\,\bar\al_X=\bar \al\}$. 
\end{claim}  
This fact easily implies that ${\mc M}_{<\omega}$ is $\sig$-compact.   We prove now the claim. The sets $\mc A_{\bar{\al}}$ are closed under isomorphic images because  if  $f:X\to Y$ is a multi-isomorphism, then $\overline{\al}_X= \overline{\al}_Y$, because $$f( \bigcap_{k\in s} \ker \nrm{\cdot}_{X,k})= \bigcap_{k\in s} \ker \nrm{\cdot}_{Y,k} \text{ for every $s\subseteq \la_X=\la_Y$.}$$ This follows from the fact that $f$ is multi-bounded. The converse is also true: for suppose that $\overline{\al}_X=\overline{\al}_Y$ and let $\lambda = \lambda_X = \lambda_Y$. First, observe that $\al_\varnothing^X=\dim X$, so we have that $\dim X=\dim Y$. The argument is by induction on  the number $N=N_X=N_Y$ of non-empty $s\subseteq \la$ such that $\al_s^X>0$. If $N=0$, this means that all seminorms are in fact norms, so the desired result follows from the well-known fact that the norms on   finite-dimensional spaces of the same dimension are all equivalent. Suppose now that   $N>0$. Let us choose $k_0<\la$ such that $\al_{\{k_0\}}>0$. For suppose that $X,Y\in \mc A_{\overline{\al}}$. Let $X_0:=\ker \nrm{\cdot}_{X,k}$, $Y_0:=\ker \nrm{\cdot}_{Y,k}$, and let $X_1\subseteq X$ and $Y_1\subseteq Y$ be complementary subspaces of $X_0$ and $Y_0$ respectively. Then $N_{X_1}=N_{Y_1}<N$, so there is some multi-isomorphism $f: X_1\to Y_1$, and similarly $X_0':=(X_0, (\nrm{\cdot}_{X, k})_{k<\la_X, k\neq k_0})$ and $Y_0':=(Y_0, (\nrm{\cdot}_{Y, k})_{k<\la_X, k\neq k_0})$ satisfy that $N_{X_0'}=N_{Y_0'}<N$, so there is a   multi-isomorphism $g: X_0'\to Y_0'$. Then $h: X\to Y$ defined by $h(x_0+x_1):= f(x_0)+g(x_1)$ for $x_0\in X_0$ and $x_1\in X_1$ is a   multi-isomorphism. 
\end{proof}  

Given $\varep \geq 0$, let $\Emb_\varep(X,Y)$ denote the set of all injective $\emph{multi-$\varep$-isometric embeddings}$ from $X$ into $Y$, i.e. linear mappings $f: X \rightarrow Y$ such that $$\frac{1}{1+\varep} \|x\|_{X,m} \leq \|f(x)\|_{Y,m} \leq (1+\varep) \|x\|_{X,m}$$ for each $x \in X$ and $m < \la_X$. When $\varep = 0$, we simply refer to such a mapping as a \emph{multi-isometric embedding}; the collection of all such mappings $f : X \rightarrow Y$ will be denoted $\Emb(X,Y)$. Note that, unlike in the setting of normed spaces, a seminorm-preserving mapping need not be injective. Also, a seminorm-preserving mapping $f : X \rightarrow Y$ need not be continuous when $\la_X < \la_Y$. A \emph{multi-isometry} (resp. multi-$\varep$-isometry) is a surjective multi-isometric embedding (resp. multi-$\varep$-isometric embedding).

Below we will make use of the notion of a \emph{modulus} relative to a set $S$, i.e. a function $\varpi: S \times [0,\infty[\to [0,\infty[$ such that, for every $s \in S$, the function $\varpi(s, \cdot)$ is increasing and continuous at $0$ with value $0$. In our case, $S$ will be the product $\N \times (\N \cup \{\omega\})$.

\begin{defi}\label{Fraisse}
Let $\K$ be a class of finite-dimensional multi-seminormed spaces.
	\begin{enumerate}[(1)]
		\item $\K$ has the \emph{hereditary property} (HP) if $X \in \K$ whenever $Y \in \K$ and $\Emb(X,Y) \neq \varnothing$.
		\item $\K$ has the \emph{joint embedding property} (JEP) if for every $X, Y \in \K$ there is $Z \in \K$ such that $\Emb(X,Z)$ and $\Emb(Y,Z)$ are non-empty.
		\item $\K$ has the \emph{near amalgamation property} (NAP) if for every $\varep > 0$, every $X,Y,Z \in \K$ and every pair of multi-isometric embeddings $f_0 : X \rightarrow Y, f_1 : X \rightarrow Z$, there is $W \in \K$ together with multi-isometric embeddings $g_0 : Y \rightarrow W, g_1 : Z \rightarrow W$ such that $\|g_0 \circ f_0 - g_1 \circ f_1\|_n \leq \varep$ for each $n < \la_X$.
		\item $\K$ is an \emph{amalgamation class with modulus of stability $\varpi$} if $(\{0\}, \|\cdot\|) \in \K$ and for every $d \in \N$, $l \in \N \cup \{\omega\}$, $\varep > 0$ and $\de\ge 0$, every $X,Y,Z \in \K$ such that $\dim X = d$ and $\la_X = l$, and every pair of multi-$\delta$-isometric embeddings $f_0 : X \rightarrow Y, f_1 : X \rightarrow Z$, there is $W \in \K$ together with multi-isometric embeddings $g_0 : Y \rightarrow W, g_1 : Z \rightarrow W$ such that:
		\begin{enumerate}[(i)]
			\item $\|g_0 \circ f_0 - g_1 \circ f_1\|_n \leq \varpi(d,l,\de)+ \varep$ for each $n < \lambda_X$.
			\item $W$ is separated.
		\end{enumerate}
		\item $\K$ is \emph{Fra\"iss\'e} if it is a hereditary $d_{\mathrm{BM}}$-closed amalgamation class such that $\K \subseteq \M_{<\omega}$.
	\end{enumerate}
\end{defi}

Observe that every amalgamation class $\K$ automatically has the JEP since one can simply amalgamate over the trivial normed space. We remark that condition (4)(ii) does not appear in the usual presentation of metric Fra\"iss\'e theory (and in particular the Fra\"iss\'e theory of Banach spaces developed in \cite{FLMT}). However, it appears to be necessary when working with arbitrary multi-normed spaces in order to guarantee injectivity of certain seminormed-preserving mappings encountered when constructing a Fra\"iss\'e limit. This condition will be trivially satisfied by all collections $\K$ of interest. Conditions (4)(i) and (4)(ii) together imply that for any $X$ belonging to an amalgamation class $\K$, there is some separated $Y \in \K$ such that $\Emb(X,Y) \neq \varnothing$. We will make use of this fact without reference.

We now proceed to show that the classes considered above are amalgamation classes. The following is a particular case of a more general result explained in Proposition \ref{amalgamation}, where it is shown that classes of multi-seminormed spaces defined from amalgamation classes of normed spaces are also amalgamation classes. We present the proof here since a version of the relevant construction (which is essentially a pushout in the category of multi-seminormed spaces) will be used in the sequel.

\begin{prop}\label{NAP}
For each $\K \in \{\mc M,\G\}$ and each $\alpha \leq \omega$, the classes $\K_{=\alpha}^{\mathrm{sep}}, \K_{\leq \alpha}^{\mathrm{sep}}$ and $\K_{<\alpha}^{\mathrm{sep}}$ are amalgamation classes with modulus of stability $\varpi(d,l,\de) = 2\de$. Furthermore, $\K_{<\omega}$ is an amalgamation class with the same modulus of stability.
\end{prop}

\begin{proof}
First we will show that $\K_{\leq \al}^{\mathrm{sep}}$ is an amalgamation class with modulus $\varpi(d,l,\de) = \de$ with respect to \emph{expansive} multi-$\delta$-isometric embeddings, i.e. multi-$\delta$-isometric embeddings $f: X \rightarrow Y$ which additionally satisfy $\|f(x)\|_{Y,n} \geq \|x\|_{X,n}$ for each $x \in X$ and $n < \la_X$. To this end, fix $\varep > 0, \delta \geq 0, d \in \N$ and $l \in \N \cup \{\omega\}$. Let $X, Y, Z \in \langle \bar{\K}\rangle$ where $\dim X = d$ and $\la_X = l$, and fix expansive multi-$\delta$-isometric embeddings $f : X \rightarrow Y$ and $g : X \rightarrow Z$. We can suppose, without loss of generality, $\lambda_Y \leq \lambda_Z$. Consider the sum $Y \oplus Z$ together with the canonical inclusion mappings $i : Y \rightarrow Y \oplus Z$ and $j : Z \rightarrow Y \oplus Z$. Equip $Y \oplus Z$ with a sequence $(\|\cdot\|_n)_{n < \lambda_Z}$ of seminorms in the following way:
\begin{enumerate}[(i)]
	\item For each $n \in [0, \la_X)$, let $$\|(y,z)\|_n := \inf\{\|u\|_{Y,n} + \|v\|_{Z,n} + (\delta+\varep) \|x\|_{X,n} : \, x \in X, u \in Y, v \in Z,  y = u + f(x), z = v - g(x)\}.$$
	\item Suppose $n \in [\la_X, \lambda_Y)$. If $\K \neq \G$, let $\|(y,z)\|_n := \max\{\|y\|_{Y,n}, \|z\|_{Z,n}\}$. Otherwise, assuming inductively that $\|(y,z)\|_{n-1}$ has been defined, let $\|(y,z)\|_n := \max\{\|(y,z)\|_{n-1}, \|y\|_{Y,n}, \|z\|_{Z,n}\}$.
	\item Suppose $n \in [\lambda_Y, \lambda_Z)$. If $\K \neq \G$, let $\|(y,z)\|_n := \|z\|_{Z,n}$. Otherwise, assuming inductively that $\|(y,z)\|_{n-1}$ has been defined, let $\|(y,z)\|_n := \max\{\|(y,z)\|_{n - 1}, \|z\|_{Z,n}\}$.
\end{enumerate}
Note that the alternative in the last two parts of the above construction is needed to guarantee that $(\|\cdot\|_n)_{n < \lambda_Z}$ is a graded sequence of seminorms whenever $X,Y$ and $Z$ are graded. Furthermore, $Y \oplus Z$ is separated whenever $Y$ and $Z$ are both separated. In the case where we are working in $\K_{<\omega}$, the above space can be turned into a separated space by extending the above sequence of seminorms by a norm.

First we check that $i$ is a multi-isometric embedding. Suppose first that $n < \la_X$. Then for each $y \in Y$, $\|i(y)\|_n = \|(y,0)\|_n \leq \|y\|_{Y,n}$ by definition of $\|\cdot\|_n$. On the other hand, given any $x \in X, u \in Y$ and $v \in Z$ such that $y = u+ f(x)$ and $0 = v - g(x)$, we have
\begin{equation*}
\begin{split}
\|u\|_{Y,n} + \|v\|_{Z,n} + (\delta+\varep) \|x\|_{X,n} &= \|u\|_{Y,n} + \|g(x)\|_{Z,n} + (\delta+\varep) \|x\|_{X,n} \\
& \geq \|u\|_{Y,n} + \|x\|_{X,n} + (\delta+\varep) \|x\|_{X,n}\\
& = \|u\|_{Y,n} + (1+\delta+\varep)\|x\|_{X,n}\\
& \geq \|u\|_{Y,n} + \|f(x)\|_{Y,n}\\
& \geq \|u + f(x)\|_{Y,n} = \|y\|_{Y,n}
\end{split}
\end{equation*}
and so $\|i(y)\|_n \geq \|y\|_{Y,n}$. Thus $i$ preserves the first $\la_X$ seminorms. If $\la_X \leq n < \lambda_Y$, then $\|i(y)\|_n = \|y\|_{Y,n}$ by definition and so $i$ is a multi-isometric embedding.

Next we check that $j$ is a multi-isometric embedding. The case where $n < \la_X$ is exactly as before. In the case where $\la_X \leq n < \lambda_Y$, we have $\|j(z)\|_n = \|z\|_{Z,n}$ by definition. Otherwise, $n \geq \lambda_Y$. In the graded case we have $$\|j(z)\|_n = \max\{\|(0,z)\|_{\lambda_Y-1}, \|z\|_{Z,n}\} = \max\{\|z\|_{Z,\lambda_Y - 1}, \|z\|_{Z,n}\} = \|z\|_{Z,n}$$ since $Z$ is graded. In the non-graded case we simply have $\|j(z)\|_n = \|z\|_{Z,n}$ by definition, and so $j$ is a multi-isometric embedding.

It remains to check $\|i \circ f - j \circ g\| \leq \delta+\varep$. To see this, note that for each $x \in X$ and $n < \la_X$ we have $$\|i(f(x)) - j(g(x))\|_n = \|(f(x), -g(x))\|_n \leq (\delta+\varep) \|x\|_{X,n}$$ by taking $u = 0$ and $v = 0$ in the definition of $\|\cdot\|_n$. Thus $i : Y \rightarrow Y \oplus Z$ and $j : Z \rightarrow Y \oplus Z$ are the desired multi-isometric embeddings.

Now suppose that $f$ and $g$ are merely multi-$\delta$-isometric embeddings. Define a new sequence of seminorms $(\|\cdot\|'_{X,n})_{n<\la_X}$ on $X$ by setting $\|x\|'_{X,n} := \frac{1}{1+\delta} \|x\|_{X,n}$. Then $f$ and $g$ become expansive multi-$\delta'$-isometric embeddings from $X$ (equipped with the new sequence of seminorms) into $Y$ and $Z$, respectively, where $\delta' = 2\delta + \delta^2$. Thus we can apply the above argument to find $W \in \K_{\leq \alpha}^{\mathrm{sep}}$ and multi-isometric embeddings $i : Y \rightarrow W, j : Z \rightarrow W$ such that $$\|i(f(x)) - j(g(x))\|_{W,n} \leq (\delta' + \varep) \|x\|'_{X,n} \text{ for all $n < \la_X$ and $x \in X$ such that $\|x\|'_{X,n} \leq 1$}.$$ Now fix $n < \la_X$ and $x \in X$ such that $\|x\|_{X,n} \leq 1$. Then $\|x\|'_{X,n} \leq \frac{1}{1+\delta} \leq 1$ and so $$\|i(f(x)) - j(g(x))\|_{W,n} \leq (\delta' + \varep) \|x\|'_{X,n} \leq \frac{\delta'}{1+\delta} + \varep = \frac{\delta + \delta(1+\delta)}{1+\delta} + \varep \leq 2\delta + \varep.$$ This shows that $\K_{\leq \al}^{\mathrm{sep}}$ is an amalgamation class with modulus $\varpi(d,l,\delta) = 2\delta$. In the case where $\K = \G$, note that the space $Y \oplus Z$ constructed above is graded whenever $X, Y$ and $Z$ are graded. This proofs for $\K_{<\alpha}^\mathrm{sep}, \K_{=\alpha}^\mathrm{sep}$ and $\K_{<\omega}$ are similar.
\end{proof}

\begin{rem}
The above proof can easily be adapted to prove amalgamation properties for classes of other ``norm-like'' structures from functional analysis. For instance, one can similarly show that the class of all finite-dimensional F-(semi)-normed spaces (as in \cite{K77}) is an amalgamation class with respect to $\delta$-isometric embeddings.
\end{rem}

Another source of examples of amalgamation classes of multi-seminormed spaces will come from the following proposition, which will make use of known amalgamation properties of classes of normed spaces. (The relevant definitions are analogous to our case; the reader is referred to \cite{FLMT} for more details, examples and discussion.) The following is standard notation: Given a seminormed space $X=(X,\nrm{\cdot})$, we denote by $ X_{\nrm{\cdot}}$ the normed space $(X/\ker\nrm{\cdot}, \nrm{\cdot})$ where $\|[x]\| := \|x\|$. Note that this norm is well-defined.

\begin{defi}\label{classes}
Let $\bar{\K}:=\{\K_n\}_{n< \al}$, $\al\le \omega$, be a collection of classes $\K_n$ of normed spaces. We define $\langle \bar{\K}\rangle_{\le \al}$, $\langle \bar{\K}\rangle_{< \al}$, $\langle \bar{\K}\rangle_{= \al}$ as the classes of separated multi-seminormed spaces $(X,(\nrm{\cdot}_n)_{n<\la_X}) \in \mc M_{\le \al}$, $\mc M_{< \al}$ and in $\mc M_{= \al}$, respectively, such that $X_{\|\cdot\|_n} \in \K_n$ for all $n < \lambda_X$.
\end{defi}  

\begin{prop}\label{amalgamation}
Let $\{\K_n\}_{n < \omega}$ be a sequence of families of finite-dimensional normed spaces.
	\begin{enumerate}[(a)]
		\item If each $\K_n$ is hereditary, then so are $\langle \bar{\K}\rangle_{< \al}$ and $\langle \bar{\K}\rangle_{\leq \al}$ for each $\alpha \leq \omega$.
		\item Suppose each $\K_n$ is an amalgamation class with modulus of stability $\varpi^{\K_n} : \N \rightarrow [0,\infty[$. Then for each $\alpha \leq \omega$ the class $\langle \bar{\K}\rangle_{< \al}$ is also an amalgamation class with modulus of stability $$\varpi(d,l,\delta) := \max_{n<l} \varpi^{\K_n}(d,\delta).$$ Furthermore, for each $n < \omega$ the classes  $\langle \bar{\K}\rangle_{\leq n}$ and $\langle \bar{\K}\rangle_{= n}$ are amalgamation classes with the same modulus of stability. 
	\end{enumerate}
\end{prop}

\begin{proof}
The proof of (a) is relatively straightforward, so we leave the details to the reader. We only check (b) for $\langle \bar{\K}\rangle_{< \al}$ since the proofs for the other classes are similar. Fix $\varep > 0, \delta \geq 0$ together with integers $d, l$. Let $X, Y, Z \in \langle \bar{\K}\rangle_{< \al}$ where $\dim X = d$ and $\la_X = l$, and fix multi-$\delta$-isometric embeddings $f_0 : X \rightarrow Y$ and $g_0 : X \rightarrow Z$. Suppose, without loss of generality, $\lambda_Y \leq \lambda_Z$. To construct the required element of $\langle \bar{\K}\rangle_{< \al}$, we will first define a sequence $(W_i)_{i < \lambda_Z}$ of finite-dimensional normed spaces. For each integer $i < l$, define mappings $f_0^i : X_{\nrm{\cdot}_i} \rightarrow Y_{\nrm{\cdot}_i}$ and $g_0^i : X_{\nrm{\cdot}_i} \rightarrow Z_{\nrm{\cdot}_i}$ by setting $$f_0^i([x]_i) = [f_0(x)]_i \text{ and } g_0^i([x]_i) = [g_0(x)]_i.$$ Then each $f_0^i$ and $g_0^i$ is a $\delta$-isometric embedding between elements of $\K_i$ and so, since $\K_i$ is an amalgamation class for each $i < n$, there are $W_i \in \K_i$ and isometric embeddings $f_1^i : Y_{\nrm{\cdot}_i} \rightarrow W_i, g_1^i : Z_{\nrm{\cdot}_i} \rightarrow W_i$ such that $$\|f_1^i \circ f_0^i - g_1^i \circ g_0^i \|_{X_{\nrm{\cdot}_i}, W_i} \leq \varpi^{\K_i}(d,\delta) + \varep.$$ For $i \in [l, \lambda_Y)$, let $W_i$ be obtained from an application of the JEP to the pair $Y_{\nrm{\cdot}_i}, Z_{\nrm{\cdot}_i}$; let $f_1^i$ and $g_1^i$, respectively, denote the corresponding multi-isometric embeddings. Finally, for $i \in [\lambda_Y, \lambda_Z)$, simply let $W_i = Z_{\nrm{\cdot}_i}$.

Let $W = \prod_{i< \lambda_Z} W_i$ and equip $W$ with a family of seminorms defined by setting $$\|(w_0, \dots, w_{\lambda_Z - 1})\|_i = \|w_i\|_{W_i} \text{ for $i < \lambda_Z$}.$$ Observe that $W \in \langle \bar{\K}\rangle_{< \al}$ since $W$ is separated and $W_{\|\cdot\|_i} \cong W_i$ for each $i < \lambda_Z$. Define $f_1 : Y \rightarrow W$ and $g_1 : Z \rightarrow W$ by $$f_1(y) = \left(f_1^0([y]_0), \dots, f_1^{\lambda_Y - 1}([y]_{\lambda_Y - 1}), 0, \dots, 0\right),$$ $$g_1(z) = \left(g_1^0([z]_0), \dots, g_1^{\lambda_Y-1}([z]_{\lambda_Y-1}), [z]_{\la_Y}, \dots, [z]_{\la_Z-1} \right).$$ Then it is straightforward to check that $f_1$ and $g_1$ are multi-isometric embeddings which witness the amalgamation property for $\K$ for the above parameters. Injectivity follows from the fact that $Y$ and $Z$ are separated.
\end{proof}

Combining our two sources of examples of amalgamation classes yields the following list:

\begin{ex}\label{examples}
\begin{enumerate}[(1)]
	\item The class $\M_{<\omega}$ of all finite-dimensional multi-seminormed spaces $X$ such that $\la_X < \omega$ is a Fra\"iss\'e class. In this case, observe that we do not need to restrict to separated multi-seminormed spaces since an arbitrary finite sequence of seminorms can always be extended by a norm. The same is true of the subclass $\G_{<\omega}$ of all graded $X \in \M_{<\omega}$.
	\item For each $\alpha \leq \omega$, the classes $\M_{<\alpha}^{\mathrm{sep}}, \M_{\leq \alpha}^{\mathrm{sep}}$ and $\M_{=\alpha}^{\mathrm{sep}}$ of all finite-dimensional \emph{separated} multi-seminormed spaces such that $\la_X < \alpha, \la_X \leq \alpha$ and $\la_X = \alpha$, respectively, are amalgamation classes. Note, however, that these are not Fra\"iss\'e classes since it is possible for a non-separated multi-seminormed space to embed into a separated space according to our definition of a multi-isometric embedding. The same is true for the corresponding graded versions $\G_{<\alpha}^{\mathrm{sep}}, \G_{\leq \alpha}^{\mathrm{sep}}$ and $\G_{=\alpha}^{\mathrm{sep}}$.
	\item For each infinite $I \subseteq \omega$, let $\M_{<\omega}^I$ denote the collection of all $(X, (\|\cdot\|_{X,n})_{n<\lambda_X}) \in \M_{<\omega}$ such that $\|\cdot\|_{X,n}$ is a norm for each $n \in I \cap [0, \lambda_X[$. Then $\M_{<\omega}^I$ is a Fra\"iss\'e class. (This follows from the proof of Proposition \ref{NAP}.) The amalgamation classes $\M_{< n}^I$ and $\M_{=n}^I$ can be defined similarly.
	\item Let $\M_{<\omega}^\mathcal{H}$ denote the class of all finite-dimensional Fr\'echet-Hilbert spaces with a finite sequence of seminorms, where $(X,(\|\cdot\|)_{n<\la_X})$ is a \emph{Fr\'echet-Hilbert space} if $\|\cdot\|_n$ is a Hilbertian seminorm, i.e. a seminorm induced by a semi-inner product on $X$. Equivalently, $X$ is Fr\'echet-Hilbert if each quotient space $X_{\|\cdot\|_n}$ is a Hilbert space. Then $\M_{<\omega}^\mathcal{H}$ is a Fra\"iss\'e class. (We refer the reader to \cite{MV} for more on Fr\'echet-Hilbert spaces and related concepts.)
	\item Given a sequence $(p_n)_{n<\omega} \subseteq [1,\infty[$ with $p_n\notin  \{4,6,8,\dots\}$, consider the class $\M_{<\omega}^{(p_n)}$ of all multi-seminormed spaces $(X,(\nrm{\cdot}_n)_{n<\la_X})$ such that for each $n<\la_X$ the normed space $X_{\nrm{\cdot}_n}$ can be isometrically embedded into $L_{p_n}[0,1]$. Then this is a Fra\"iss\'e class for any such sequence $(p_n)$. This follows from the fact that $\Age(L_p[0,1])$ is a Fra\"iss\'e class of finite-dimensional normed spaces for each $p \in [1,\infty[$ such that $p \not \in \{4,6,8,\dots\}$ (see \cite{FLMT}). In this case, the corresponding modulus of stability $\varpi$ associated to $\Age(L_p[0,1])$ depends on the dimension for each $p$, and hence so does the corresponding modulus of stability associated to $\M_{<\omega}^{(p_n)}$.
	\item In general, one can combine the various classes of normed spaces considered above to form new amalgamation classes of multi-seminormed spaces. For instance, given any sequence $(E_n)_{n<\alpha}$ of Fra\"iss\'e Banach spaces, one can set $\K_n = \Age(E_n)$ for each $n$ (noting that these are all Fra\"iss\'e classes of normed spaces) and form the associated class $\langle \bar{\K}\rangle_{\le \alpha}$ of multi-seminormed spaces, where $\bar{\K} = (\K_n)$.
\end{enumerate}  
\end{ex}

In all but the last two examples, the corresponding modulus of stability is always independent of both the ``dimension'' and the ``length'' parameters in the sense that modulus only depends on the third variable. In the setting of normed spaces, such classes have been studied in \cite{Lupini2018} and are called \emph{stable} Fra\"iss\'e classes. Thus, stable Fra\"iss\'e classes of finite-dimensional normed spaces give rise to stable Fra\"iss\'e classes of finite-dimensional multi-seminormed spaces. In the last two families of examples, however, it is unknown if the modulus of stability depends on the dimension. This observation inspires the following:

\begin{que}
Are there any examples of amalgamation classes of finite-dimensional seminormed spaces such that the corresponding modulus of stability depends on all three parameters?
\end{que}

\section{Fra\"iss\'e Fr\'echet spaces}

For a multi-seminormed space $X$, let $\Iso(X)$ be the group of multi-isometries $f: X \rightarrow X$ equipped with the topology generated by basic open sets of the form $$\{g \in \Iso(X) : \max_{m < n} \|(f -  g)\restriction Y\|_m < \varep \}$$ where $f \in \Iso(X)$, $Y$ is a finite-dimensional linear subspace of $X$, $\varep > 0$ and $n < \la_X$. This is the analogue of the strong operator topology in the setting of multi-seminormed spaces; in particular, a sequence $(f_k)$ in $\Iso(X)$ converges to $f$ if and only if $\|f_k(x) - f(x)\|_n$ converges to $0$ as $k \rightarrow \infty$ for each $x \in X$ and each $n \in \N$. For the sake of brevity, we will simply refer to this topology as the topology of pointwise convergence on $\Iso(X)$. We also equip $\Iso(X)$ with its \emph{left uniform structure}, i.e. the uniformity generated by entourages of the diagonal of the form $$\{(f,g) \in \Iso(X)^2 : f^{-1}\circ g \in U\}$$ where $U$ is a neighbourhood of the identity.

For a multi-seminormed space $(E, (\|\cdot\|_{E,n})_{n < \lambda_E})$, let $\Age(E)$ denote the \emph{age of $E$}, which is defined as the class of all finite-dimensional multi-seminormed spaces of the form $$(X, (\|\cdot\|_{E,n})_{n < \la_X})$$ where $X$ is a linear subspace of $E$ and $\la_X \leq \lambda_E$. (Warning: This is not the standard definition of the age, which is usually defined as the collection of all finitely-generated substructures of a given structure.) We also let $\Age_\alpha(E) := (\Age(E))_{=\alpha} = \Age(E) \cap \M_{=\alpha}$ for each $\alpha \leq \omega$. The collections $\Age_{\leq \alpha}(E)$ and $\Age_{\leq \alpha}(E)$ are defined similarly.

\begin{defi}
Let $(E, (\|\cdot\|_n)_{n < \lambda_E})$ be a multi-seminormed space and let $\K$ be a class of finite-dimensional multi-seminormed spaces.
	\begin{enumerate}[(a)]
		\item $E$ is \emph{$\K$-universal} if every $X \in \K$ embeds into $E$ via a multi-isometric embedding.
		\item $E$ is \emph{approximately $\K$-ultrahomogeneous} if for every $\varep > 0, X \in \K$ and $\gamma, \eta \in \Emb(X,E)$, there is $g \in \Iso(E)$ such that $\| g \circ \gamma - \eta \|_n  \leq \varep$ for each $n < \la_X$.
		\item $E$ is \emph{$\K$-Fra{\"i}ss\'e with modulus of stability $\varpi$} if for every $d \in \N, l \in \N \cup \{\omega\}, \varep > 0$ and $\delta \geq 0$, every $X \in \K$ and every $\gamma, \eta \in \Emb_\delta(X,E)$, there is $g \in \Iso(E)$ such that $\| g \circ \gamma - \eta \|_n  \leq \varpi(d, l, \delta) + \varep$ for each $n < l$.
	\end{enumerate}
When $\K = \Age(E)$, we will occasionally omit the reference to $\K$ in the above definitions. We also omit the reference to the modulus of stability when working with a general Fra\"iss\'e Fr\'echet space.
\end{defi}

In order to transfer properties of an amalgamation class to a relevant Fra\"iss\'e Fr\'echet space, we will need a slight modification of the given modulus. To this end, for a modulus $$\varpi : \N \times (\N \cup \{\omega\}) \times [0,\infty[ \rightarrow [0,\infty[$$ we define a new modulus $$\varpi^*(d,l,\delta) := \inf_{\delta' > \delta} \varpi(d,l,\delta').$$ Note that $\varpi^*$ is indeed a modulus which furthermore satisfies $\varpi(d,l,\delta) \leq \varpi^*(d,l,\delta)$ for any $d \in \N$ and $l \in \N \cup \{\omega\}$.

Our main goal in this section is to show that when $\K$ is a Fra\"iss\'e class with modulus $\varpi$, then there is a unique (up to a multi-isometry) separable, $\K$-universal and $\K$-Fra\"iss\'e Fr\'echet space with modulus $\varpi^*$ whose finite-dimensional subspaces are precisely those from $\K$. Note that any $\K$-Fra\"iss\'e space is automatically approximately $\K$-ultrahomogeneous. Furthermore, if $E$ is Fra\"iss\'e, then $\Age_{<\omega}(E)$ is a Fra\"iss\'e class; the amalgamation property follows from the Fra\"iss\'e property together with the fact that any finite-dimensional subspace $X \in \Age_{<\omega}(E)$ is eventually separated by the sequence of seminorms from $E$.

Below we will need to make use of an inductive limit construction for spaces with a finite sequence of seminorms. Suppose $(X_n, I_n)_{n < \omega}$ is a sequence such that:
	\begin{enumerate}[(1)]
		\item $X_n \in \M_{<\omega}$ for each $n$.
		\item $(\lambda_{X_n})_{n < \omega}$ is non-decreasing and converges to a fixed $\lambda \leq \omega$.
		\item $I_n \in \Emb(X_n, X_{n+1})$ for each $n$.
	\end{enumerate}
We define the \emph{inductive limit} $\lim_n (X_n, I_n)$ as follows: First, for each $m \leq n$, define $I_{m,n} \in \Emb(X_m, X_n)$ by setting $I_{m,m} = \id_{X_m}$ and $I_{m,n+1} = I_n \circ I_{m,n}$. Then let $V$ be the linear subspace of the product space $\prod_n X_n$ defined by declaring $(x_n)_{n < \omega} \in V$ if and only if there is some $m$ such that $x_n = I_{m,n}(x_m)$ for all $n \geq m$. For each $k < \lambda$, let $N_k$ be the linear subspace of $V$ consisting of all $(x_n)_n$ such that there is $m$ such that $k < \lambda_{X_m}$ and $\|x_n\|_k = 0$ for all $n \geq m$. Let $N = \bigcap_{k < \omega} N_k$ and let $V_0 = V / N$.

Define a sequence of seminorms $(\|\cdot\|_k)_{k < \lambda}$ on $V_0$ by $$\|(x_n)_n + N\|_k = \|x_m\|_{X_m, k}$$ where $m$ is the least integer such that $k < \lambda_{X_m}$ and $x_n = I_{m,n}(x_m)$ for all $n \geq m$. Note that these seminorms are well-defined and they form an increasing sequence whenever each $X_n$ is graded since each $I_n$ is a multi-isometric embedding.

To take a completion, we proceed as in the case of normed spaces. To this end, define an equivalence relation $\sim$ on the space of all Cauchy sequences in $V_0$ by declaring $$(\sigma_n) \sim (\tau_n) \iff (\forall k < \lambda) \, \lim_n \|\sigma_n - \tau_n\|_k = 0,$$ where a sequence $(\sigma_n) \subseteq V_0$ is Cauchy if it is Cauchy with respect to each seminorm $\|\cdot\|_k$ on $V_0$. Let $\lim_n (X_n, I_n)$ denote the resulting quotient space equipped with the sequence of seminorms of length $\lambda$ given by $$\vertiii{[(\sigma_n)]_\sim}_k = \lim_n \|\sigma_n\|_k \text{ for each $k < \lambda$.}$$ Observe that $(\lim_n (X_n, I_n), (\vertiii{\cdot}_k)_{k<\lambda})$ is a complete, separated multi-seminormed space, and so it is a Fr\'echet space. Given $x \in V_0$, let $\mathcal{C}(x)$ denote the equivalence class (in $\lim_n (X_n, I_n)$) of the Cauchy sequence with constant value $x$. Then for each $m$ the canonical mapping $$I_m^{(\infty)} : X_m \rightarrow \lim_n (X_n, I_n) : x \mapsto \mathcal{C}\big((\overset{(m)}{\overbrace{0,\dots, 0}}, x, I_{m,m+1}(x), I_{m,m+2}(x), \dots)\big)$$ is a seminorm-preserving linear mapping. When $X_m$ is separated, it is routine to check that $I_m^{(\infty)}$ is injective and hence a multi-isometric embedding. Note that union $\bigcup_n I_n^{(\infty)}(X_n)$ is dense in $\lim_n (X_n, I_n)$, in the sense that for every $[\sigma]_\sim \in \lim_n (X_n,I_n)$, there is a sequence $(\sigma_m)$ belonging to the union which converges to $[\sigma]_\sim$ with respect to each $\vertiii{\cdot}_k$. In particular, such a sequence converges to $[\sigma]_\sim$ with respect to the pseudometric induced by $\max_{l < k} \vertiii{\cdot}_l$ for any given $k < \lambda$. From now on, whenever we are working with an inductive limit of separated spaces $X_n$, we will identify $X_n$ with its image $I_n^{(\infty)}(X_n)$ in $\lim_n(X_n, I_n)$. In this way, each $I_n : X_n \rightarrow X_{n+1}$ becomes the corresponding inclusion mapping, so that $(X_n)$ is an increasing sequence of finite-dimensional subspaces of $\lim_n(X_n,I_n)$ and $X_n$ is equipped with the first $\lambda_{X_n}$ seminorms induced by the inductive limit. Furthermore, $\bigcup_n X_n$ is dense in the inductive limit.

We will need the following consequence of the amalgamation property, the proof of which can be found in \cite[Lemma 2.32]{FLMT}.

\begin{lem}\label{AP}
Suppose $\K$ is an amalgamation class with modulus $\varpi$. Fix $\varep > 0$, $\Delta \subseteq \R^+$ finite and $\A \cup \{Y\} \subseteq \K$ finite. Then there is some $Z \in \K$ and $I \in \Emb(Y,Z)$ such that for every $X \in \A$, every $\delta \in \Delta$ and every $\gamma, \eta \in \Emb_\delta	(X,Y)$, there is $J \in \Emb(Y,Z)$ such that $$\max_{l < \la_X} \| I \circ \gamma - J \circ \eta \|_l \leq  \varpi(\dim X, \la_X, \delta) + \varep.$$
\end{lem}

Before proceeding, we require one more piece of notation. Given two classes $\K$ and $\K'$ of finite-dimensional multi-seminormed spaces, write $\K \preceq \K'$ when for every $X \in \K$ there are $Y \in \K'$ and a multi-isometry of $X$ onto $Y$, and write $\K \equiv \K'$ when $\K \preceq \K' \preceq \K$.

\begin{thm}\label{fraisse}
If $\K$ is an amalgamation class with modulus $\varpi$, then there is a separable $\K$-Fra\"iss\'e Fr\'echet space $E$ with modulus $\varpi^*$ such that $\K \preceq \Age_{<\omega}(E)$. Furthermore, if $\K$ is a Fra\"iss\'e class then $\Age_{<\omega}(E) \equiv \K$.
\end{thm}
\begin{proof}
Let $\{Z_n\}_{n<\omega}$ be a countable $d_{\mathrm{BM}}$-dense subset of $\K$. Fix an enumeration $(\delta_n)$ of $\Q \cap [0,1]$ such that $\delta_0 = 0$. Using Lemma \ref{AP}, we find a sequence $(X_n, I_n)_{n<\omega}$ of separated spaces $X_n \in \K$ and multi-isometric embeddings $I_n \in \Emb(X_n, X_{n+1})$ with the following properties:
	\begin{enumerate}[(a)]
		\item Let $\lambda_\K :=  \sup_{X \in \K} \la_X$.
		\begin{enumerate}
			\item[(i)] If $\lambda_\K = \omega$, then $\lambda_{X_n} \geq n$ for all $n$.
			\item[(ii)] If $\lambda_\K < \omega$, then $\lambda_{X_n} = \lambda_\K$ for all $n$.
		\end{enumerate}
		\item For every $k \leq n$, every $X \in \{Z_j\}_{j\leq n} \cup \{X_j\}_{j\leq n}$ and every $\gamma, \eta \in \Emb_{\delta_k}(X, X_n)$, there is $J \in \Emb(X_n, X_{n+1})$ such that $$\max_{l < \la_X} \| I_n  \circ \gamma - J \circ \eta \|_l \leq \varpi(\dim X, \la_X, 2^{-n}) + 2^{-n}.$$
		\item $\Emb(Z_m, X_n) \neq \varnothing$ for each pair $m<n$.
	\end{enumerate}
Note that (c) can be arranged by applying the JEP of $\K$. Let $E = \lim_n (X_n, I_n)$. For simplicity, we assume from now on that $(X_n)$ is an increasing sequence of subspaces of $E$, $I_n$ is the inclusion mapping $X_n \subseteq X_{n+1}$, and $I_n^{(\infty)}$ is the inclusion mapping $X_n \subseteq E$. In particular, we work exclusively with the sequence of seminorms $(\|\cdot\|_k)_{k<\lambda_\K}$ associated to $E$.

The proof of the theorem will be complete once we prove the following two claims.

\begin{claim}
$E$ is $\K$-Fra\"iss\'e.
\end{claim}
\begin{proof}[Proof of Claim]
Fix $X \in \K$ together with $\varep > 0, \delta' > \delta \geq 0$ and $\gamma, \eta \in \Emb_\delta(X,E)$. Choose a large enough $n$ such that $2^{-(n-1)} < \varep$ and such that there are $j, k \leq n$ and a sufficiently small $\delta'' \geq 0$ with the following properties:
	\begin{enumerate}[(i)]
		\item $\delta < \delta_j < \delta'$.
		\item $(\varpi(\dim X, \lambda_X, \delta_j) + \varep)\delta'' < \varep/3$.
		\item There are $\theta \in \Emb_{\delta''}(X,Z_k)$ and $\widetilde{\gamma}, \widetilde{\eta} \in \Emb_{\delta_j}(Z_k, X_n)$ such that $$\max_{l<\lambda_X} \|\widetilde{\gamma} \circ \theta - \gamma\|_l \leq \varep/3 \, \text{ and } \, \max_{l<\lambda_X} \|\widetilde{\eta} \circ \theta - \eta\|_l \leq \varep/3.$$
	\end{enumerate}
Using the definition of the sequence $(X_n, I_n)$, recursively construct sequences of embeddings $$J_s \in \Emb(X_{n+2s}, X_{n+2s+1}) \text{ and } L_s \in \Emb(X_{n+2s+1}, X_{n+2s+2})$$ such that:
	\begin{enumerate}
		\item[(iv)] $\|J_0 \circ \widetilde{\eta} - \widetilde{\gamma} \|_l \leq  \varpi(\dim X, \la_X,\delta_j) + 2^{-n}$ for each $l < \la_X$.
		\item[(v)] $\|J_{s+1} \circ L_s - \id_{X_{n+2s+1}}\|_l \leq 2^{-(n+2s)}$ for each $s\geq 0$ and each $l < \lambda_{X_{n+2s+1}}$.
		\item[(vi)] $\|L_s \circ J_s - \id_{X_{n+2s}}\|_l \leq 2^{-(n+2s+1)}$ for each $s\geq 0$ and each $l < \lambda_{X_{n+2s}}$.
	\end{enumerate}
Letting $\varep_0 = \varpi(\dim X, \la_X,\delta_j) + 2^{-n}$ and $\varep_m = 2^{-(n+m)}$ for $m\geq 1$, the above situation can be summarized by the following approximately commutative diagram:

\begin{center}
\begin{tikzcd}[column sep=large]
                      &                                                                      & X_n \arrow[r, "I_n", hook] \arrow[dd, phantom, "\scalebox{0.8}{\cca{\varep_0}}" description] & X_{n+1} \arrow[r, "I_{n+1}", hook] \arrow[rdd, "L_0"'] \arrow[dd, phantom,"\scalebox{0.8}{\cca{\varep_{1}}}" description] & X_{n+2} \arrow[r, "I_{n+2}", hook] \arrow[dd, phantom,"\scalebox{0.8}{\cca{\varep_{2}}}" description] & X_{n+3} \arrow[r, "I_{n+3}", hook] \arrow[rdd, "L_1"'] \arrow[dd, phantom,"\scalebox{0.8}{\cca{\varep_{3}}}" description] & X_{n+4} \\
X \arrow[r, "\theta"] & Z_k \arrow[ru, "\widetilde{\gamma}"] \arrow[rd, "\widetilde{\eta}"'] &                                                                                &                                                                                                            &                                                                                        &                                                                                            & \dots   \\
                      &                                                                      & X_n \arrow[ruu, "J_0"'] \arrow[r, "I_n"', hook]                                      & X_{n+1} \arrow[r, "I_{n+1}"', hook]                                                                              & X_{n+2} \arrow[r, "I_{n+2}"', hook] \arrow[ruu, "J_1"']                                      & X_{n+3} \arrow[r, "I_{n+3}"', hook]                                                                              & X_{n+4}
\end{tikzcd}
\end{center}

Now, according to (v) and (vi), for each fixed $x \in X_{n+2s}$ and $l < \lambda_{X_{n+2s}}$ we have $$\|J_{s+1} L_{s} J_s(x) - J_s(x)\|_l \leq 2^{-(n+2s)}\|J_s(x)\|_l = 2^{-(n+2s)}\|x\|_l.$$ On the other hand, we also have $$\|J_{s+1} L_{s} J_s(x) - J_{s+1}(x)\|_l = \|J_{s+1}(L_{s}J_s(x) - x)\|_l = \|L_{s}J_s(x) - x\|_l \leq 2^{-(n+2s+1)}\|x\|_l$$ and so an application of the triangle inequality then yields $$\|J_{s+1}(x) - J_s(x)\|_l \leq (2^{-(n+2s)} + 2^{-(n+2s+1)})\|x\|_l = 3\cdot 2^{-(n+2s+1)}\|x\|_l.$$ Then for every $s, t < \omega$ and every we have
\begin{equation}
\begin{split}
\|J_{s+t}(x) - J_s(x)\|_l &\leq \sum_{0 \leq i \leq t-1} \|J_{s+i+1}(x) - J_{s+i}(x)\|_l \leq \sum_{0 \leq i \leq t-1} 3\cdot 2^{-(n+2(s+i)+1)} \|x\|_l \\ 
& = \frac{3}{2^{n+2s+1}} \sum_{0 \leq i \leq t-1} 2^{-(2i)} \|x\|_l \leq \frac{3}{2^{n+2s+1}} \cdot 2 \|x\|_l \leq \frac{3}{2^{n+2s}}\|x\|_l.
\end{split}
\end{equation}
In particular, this implies that for each $x \in \bigcup_{n < \omega} X_n$ the sequence $(J_s(x))_s$ is Cauchy with respect to each seminorm $\|\cdot\|_{E,l}$ for $l < \lambda_E$; indeed, given $l < \lambda_E$, simply choose a large enough $N$ such that $l < \lambda_{X_n}$ and $x \in X_n$ for all $n > N$, noting that inequality (1) holds for all such $n$. Thus $(J_s)_{s<\omega}$ is pointwise Cauchy in $E$, so by completeness we can define a linear mapping $J : \bigcup_{n < \omega} X_n \rightarrow E$ by setting $J(x) = \lim_{s \geq k} J_s(x)$ where $k$ is least such that $x \in X_k$; we then extend $J$ to a mapping $J : E \rightarrow E$. Note that $J$ is a seminorm-preserving linear mapping. To see that $J$ is a bijection, we define as before a seminorm-preserving linear mapping $L : \bigcup_{n < \omega} X_n \rightarrow E$ by $L(y) = \lim_{s \geq k} L_s(y)$ where $k$ is least such that $y \in X_k$, and we extend it to a mapping $L : E \rightarrow E$. Then, since $E$ is separated, (v) and (vi) imply $L \circ J = J \circ L = \id_E$. Thus $J$ and $L$ are both multi-isometries. Finally, note that for each $l < \la_X$ we have $$\| J_s \circ \widetilde{\eta} - \widetilde{\gamma}\|_l \leq \varpi(\dim X, \lambda_X, \delta_j) + \sum_{0 \leq i \leq 2s} 2^{-(n+i)} \leq \varpi(\dim X, \lambda_X, \delta_j) + 2^{-(n-1)}.$$ Taking the limit as $s \rightarrow \infty$ we see that $\|J \circ \widetilde{\eta} - \widetilde{\gamma}\|_l \leq \varpi(\dim X, \lambda_X, \delta_j) + 2^{-(n-1)}$ and so
\begin{equation*}
\begin{split}
\|J \circ \eta - \gamma \|_l &\leq \|J \circ \eta - J \circ \widetilde{\eta} \circ \theta\|_l + \|J \circ \widetilde{\eta} \circ \theta - \widetilde{\gamma} \circ \theta\|_l + \|\widetilde{\gamma} \circ \theta - \gamma\|_l \\
& \leq \frac{\varep}{3} + (\varpi(\dim X, \lambda_X, \delta_j) + 2^{-(n-1)})\|\theta\|_l + \frac{\varep}{3} \\
& \leq \varpi(\dim X, \la_X, \delta_j) + \varep \leq \varpi(\dim X, \la_X, \delta') + \varep.
\end{split}
\end{equation*}
Thus $E$ is $\K$-Fra\"iss\'e with modulus $\varpi^*$.
\end{proof}

\begin{claim}
$\K \preceq \Age_{<\omega}(E)$. Furthermore, $\Age_{<\omega}(E) \preceq \K$ whenever $\K$ is Fra\"isse.
\end{claim}
\begin{proof}[Proof of Claim]
Fix $X \in \K$. Since $\K$ is an amalgamation class, $X$ isometrically embeds into a separated space $X'$. Thus we can assume without loss of generality that $X$ itself is separated. Now, find a decreasing positive sequence $(\delta_n)_n$ such that $\varpi(\dim X,\la_X,\delta_n) \leq 2^{-n}$ for each $n$. Using the definition of the sequence $\{Z_n\}$ together with property (c), for each $n$ we find some $\gamma_n \in \Emb_{\delta_n}(X,E)$. Since $E$ is $\K$-Fra\"iss\'e, for each $n$ we can choose $g_n \in \Iso(E)$ such that $\|g_n \circ \gamma_{n+1} - \gamma_n\|_l \leq 2^{-(n-1)}$ (where we take $\varep = 2^{-n}$) for each $l < \la_X$. Define a sequence $(\eta_n)_n$ of multi-$\delta_n$-isometric embeddings of $X$ into $E$ by setting $\eta_0 = \gamma_0$ and $$\eta_{n+1} = g_0 \circ \dots \circ g_n \circ \gamma_{n+1} \text{ for each $n > 0$}.$$ Then $(\eta_n)_n$ is pointwise Cauchy, since $$\|\eta_{n+k} - \eta_n\|_l \leq \sum_{j=n}^{n+k-1} 2^{-(j-1)} \leq 2^{-(n-2)}$$ and so the limit $\eta : X \rightarrow E$ is a multi-isometric embedding; injectivity follows from the assumption that $X$ is separated.

Finally, note that the construction of $E$ implies $\Age_{<\omega}(E) \preceq \overline{\K}^{\mathrm{BM}}$, where the latter collection is the $d_{\mathrm{BM}}$-closure of $\K$ in $\M_{<\omega}$. Thus, if $\K$ is a Fra\"iss\'e class, i.e. a $d_{\mathrm{BM}}$-closed amalgamation class with the hereditary property, then the latter collection is precisely $\K$. Thus $\Age_{<\omega}(E) \preceq \K$ and so $\Age_{<\omega}(E) \equiv \K$ in this case.
\end{proof}
This completes the proof of the two claims and hence of the existence of a separable, $\K$-universal, $\K$-Fra\"iss\'e Fr\'echet space.
\end{proof}

We will henceforth refer to the space $E$ constructed in Theorem \ref{fraisse} as the \emph{Fra\"iss\'e limit} of the class $\K$ and we will denote it by $\Flim(\K)$. Next we show that such a space is unique whenever $\K$ is a Fra\"iss\'e class. More generally, we have:

\begin{prop}\label{unique}
Suppose $E$ and $F$ are separable approximately ultrahomogeneous Fr\'echet spaces such that $\lambda_E = \lambda_F$ and $\Age_{<\omega}(E) \equiv \Age_{<\omega}(F)$. Then $E$ and $F$ are multi-isometric.
\end{prop}
\begin{proof}
Recursively define increasing sequences $(X_n)$ and $(Y_n)$ of elements of $\Age(E)$ and $\Age(F)$, respectively, sequences $(k_n)$ and $(l_n)$ of integers, and sequences of multi-isometric embeddings $(\gamma_n : X_n \rightarrow Y_n)$ and $(\eta_n : Y_n \rightarrow X_{n+1})$ such that the following conditions hold:
	\begin{enumerate}[(i)]
		\item $(k_n)$ and $(l_n)$ are non-decreasing and converge to $\lambda_E = \lambda_F$.
		\item $X_n \in \Age_{k_n}(E)$ and $Y_n \in \Age_{l_n}(F)$.
		\item $\|\eta_n \circ \gamma_n -  \id_{X_n} \|_{E,m} \leq 2^{-n}$ for all $m < k_n$.
		\item $\|\gamma_{n+1} \circ \eta_n - \id_{Y_n}\|_{F,m} \leq 2^{-n}$ for all $m < l_n$.
		\item $\bigcup_{n < \omega} X_n$ and $\bigcup_{n < \omega} Y_n$ are dense in $E$ and $F$, respectively.
	\end{enumerate}
To start, let $X_0 = Y_0 = \{0\}, \gamma_0 = 0$ and $k_0 = 1$. Now assume we have defined $X_n, Y_n, \gamma_n, \eta_{n-1}, k_n$ and $l_n$ for $n \geq 0$. Using the fact that $\Age_{<\omega}(E) \equiv \Age_{<\omega}(F)$, fix $\theta \in \Emb(Y_n, E)$. Since $E$ is approximately ultrahomogeneous, we can find $g \in \Iso(E)$ such that $$\|g \circ \theta \circ \gamma_n - \id_{X_n}\|_{E,m} \leq 2^{-n} \text{ for every $m < k_n$}.$$ Let $\eta_n := g \circ \theta \in \Emb(Y_n, E)$, let $k_{n+1} = \min\{l_n +1, \lambda_E\}$ and let $X_{n+1}$ be a finite-dimensional subspace of $E$ containing $X_n \cup \eta_n(Y_n)$, appropriately enlarged so that condition (v) will eventually hold, and equipped with the first $k_{n+1}$ seminorms from $E$. The construction of $Y_{n+1}, \gamma_{n+1} : Y_{n+1} \rightarrow X_{n+1}$ and $l_{n+1}$ is similar. This completes the inductive construction.

Now, note that for each $x \in \bigcup_{n < \omega} X_n$ the sequence $(\gamma_n(x))_n$ is Cauchy with respect to each seminorm $\|\cdot\|_{F,m}$: Given $m < \omega$, we can choose a large enough $N$ such that $m < \lambda_{X_n}$ and $x \in X_n$ for all $n > N$. Then $\|\gamma_{n+1}(x) - \gamma_n(x)\|_{F,m} \leq 2^{-(n-1)}$, which implies that the sequence $(\gamma_n)_{n<\omega}$ is pointwise Cauchy with respect to $\|\cdot\|_{F,m}$. Thus $(\gamma_n)_{n<\omega}$ is pointwise Cauchy in $F$, so by completeness we can define a linear mapping $\gamma : \bigcup_{n < \omega} X_n \rightarrow F$ by setting $\gamma(x) = \lim_{n \geq k} \gamma_n(x)$ where $k$ is least such that $x \in X_k$. Similarly, we can define a multi-isometric embedding $\eta : \bigcup_{n < \omega} Y_n \rightarrow E$ by $\eta(y) = \lim_{n \geq k} \eta_n(y)$ where $k$ is least such that $y \in Y_k$, and we extend it to a mapping $\eta : F \rightarrow E$. Then, since $E$ and $F$ are both separated, (iii) and (iv) imply $\gamma \circ \eta = \id_F$ and $\eta \circ \gamma = \id_E$, and so $\gamma$ is a multi-isometry.
\end{proof}

\begin{cor}[Fra\"iss\'e correspondence]
Let $\K$ be a class of finite-dimensional multi-seminormed spaces such that $\K \subseteq \M_{<\omega}$. The following are equivalent:
	\begin{enumerate}[(1)]
		\item $\K$ is a Fra\"iss\'e class.
		\item $\K \equiv \Age_{<\omega}(E)$ for a unique separable Fra\"iss\'e Fr\'echet space $E = \Flim(\K)$.
	\end{enumerate}
\end{cor}
\begin{proof}
Suppose $\K$ is a Fra\"iss\'e class. By the previous two results, $\Flim(\K)$ exists and is unique. By construction, we have $\K \equiv \Age_{<\omega}(\Flim(\K))$. To prove the other direction of the corollary, assume $E$ is s separable Fra\"iss\'e Fr\'echet space such that $\K \equiv \Age_{<\omega}(E)$. Then $\K$ is a hereditary, $d_{\mathrm{BM}}$-closed class. Furthermore, the fact that $E$ is separated implies that the class of separated elements of $\Age_{<\omega}(E)$ is cofinal in $\Age_{<\omega}$. Indeed, given $(X, (\|\cdot\|_n)_{n<m}) \in \Age_{<\omega}(E)$, we can use the fact that $X$ is finite-dimensional to find a sufficiently large $N$ such that $(X, (\|\cdot\|_n)_{n<N})$ becomes a separated subspace of $E$. From this observation, the amalgamation property of $\K$ follows from the Fra\"iss\'e property of $E$. Thus $\K$ is a Fra\"iss\'e class.
\end{proof}

\begin{ex}\label{ex}
	\begin{enumerate}[(1)]
\item Let $\GG^\omega$ be the product of countably many copies of the Gurarij space $\GG$. In \cite{BKK} it is shown that there is a sequence $(\|\cdot\|_n)_{n < \omega}$ of seminorms on $\GG^\omega$ such that $(\GG^\omega, (\|\cdot\|_n)_{n < \omega})$ is a separable graded Fr\'echet space which is Fra\"iss\'e and universal for the class of all finite-dimensional graded multi-seminormed spaces with an infinite sequence of seminorms. It is also shown that there is a sequence $(\|\cdot\|'_n)_{n < \omega}$ of seminorms such that $(\GG^\omega, (\|\cdot\|'_n)_{n < \omega})$ is a separable Fr\'echet space which is Fra\"iss\'e and universal for the class of all finite-dimensional multi-seminormed spaces with an infinite sequence of seminorms. (The authors of \cite{BKK} do not use Fra\"iss\'e-theoretic terminology; however, this follows from \cite[Proposition 4.1]{BKK} and \cite[Proposition 5.5]{BKK}, respectively.) Below we will show that these two spaces can be obtained as Fra\"iss\'e limits of the classes $\G_{<\omega}$ and $\M_{<\omega}$, respectively.
		
\item For each $n \geq 1$, the spaces $\Flim(\M_{\leq n}^{\mathrm{sep}})$ and $\Flim(\G_{\leq n}^{\mathrm{sep}})$ can be seen as separated $n$-seminormed versions of the spaces considered in the previous example. (See Example \ref{examples} for the relevant notation.) Note that $\Age_{<\omega}(\Flim(\M_{\leq n}^{\mathrm{sep}}))$ is strictly larger than $\M_{\leq n}^{\mathrm{sep}}$, since the former collection contains non-separated multi-seminormed spaces. An analogous fact holds for $\Flim(\G_{\leq n}^{\mathrm{sep}})$.

\item Let $E$ be the space $(\GG^\omega, (\|\cdot\|_n)_{n < \omega})$ considered in Example \ref{ex}(1). For each $k \in \N$, let $E_k$ be the multi-seminormed space $(\GG^\omega, (\|\cdot\|_n)_{n < k})$ obtained by truncating the associated sequence of seminorms. Then using the properties of $(\GG^\omega, (\|\cdot\|_n)_{n < \omega})$ it follows that $E_k$ is universal and Fra\"iss\'e for $\M_{<k}$. An interesting special case occurs when $n=1$, since $E_1$ can be seen as a seminormed version of the Gurarij space. In fact, if we let $\|\cdot\|$ be the seminorm on $E_1$, then the quotient $E_1 / \ker \|\cdot\|$ equipped with the corresponding quotient norm is separable, universal and approximately ultrahomogeneous for the class of all finite-dimensional Banach spaces, and so it is isometric to the Gurarij space. Note that the spaces $E_k$ are not separated, and so it is unclear if they are unique up to isometry.
		
		\item $E = \Flim(\M_{<\omega}^{\mathcal{H}})$ is the unique separable Fr\'echet-Hilbert space which is $\M_{<\omega}^{\mathcal{H}}$-Fra\"iss\'e. To see that $E$ is indeed Fr\'echet-Hilbert, observe that each quotient space $E_{\|\cdot\|}$ (where $\|\cdot\|$ is a seminorm belonging to the sequence of seminorms associated to $E$) is approximately ultrahomogeneous for the class of all finite-dimensional Hilbert spaces, and hence is isometric to a Hilbert space.
		
		\item For each sequence $(p_n) \subseteq [1,\infty[$ with $p_n\notin  \{4,6,8,\dots\}$, $\Flim(\M_{<\omega}^{(p_n)})$ is a Fra\"iss\'e Fr\'echet space. In particular, if $p_n = p$ for each $n$, then $\Flim(\M_{<\omega}^p)$ can be seen as a \emph{Fr\'echet-$L_p$-space}, which is the $L_p$ analogue of the space considered in the previous example.
	\end{enumerate}
\end{ex}

To conclude this section, we will use Proposition \ref{unique} to show that the spaces considered in Example \ref{ex}(1) can be obtained as Fra\"iss\'e limits; we will make use of the fact that $E$ is universal and approximately ultrahomogeneous for the class of all finite-dimensional (graded) multi-seminormed spaces with an infinite sequence of seminorms. From now until the end of the section, $E$ will denote one of these two spaces. Before proceeding, we need some new terminology:  Given $n < \omega$ and two multi-seminormed spaces $X$ and $Y$ such that $\la_X = \lambda_Y = \omega$, an \emph{multi-isometric $n$-embedding} from $X$ to $Y$ is an injective linear mapping $f : X \rightarrow Y$ such that $\|f(x)\|_{Y,m} = \|x\|_{X,m}$ for each $x \in X$ and each $m < n$.

The following is a version of the near amalgamation property in the context of multi-isometric $n$-embeddings for a fixed $n$.

\begin{lem}
Suppose $X, Y$ and $Z$ are finite-dimensional separated multi-seminormed spaces with $\la_X = \lambda_Y = \lambda_Z =  \omega$ and $f : X \rightarrow Y, g : X \rightarrow Z$ are multi-isometric $n$-embeddings for a fixed $n < \omega$. Then for every $\varep > 0$ there is a finite-dimensional multi-seminormed space $W$ with $\lambda_W = \omega$ and multi-isometric embeddings $i_Y : Y \rightarrow W$ and $i_Z : Z \rightarrow W$ such that $\| i_Y \circ f - i_Z \circ g \|_m \leq \varep$ for all $m < n$. Furthermore, $W \in \G$ whenever $X,Y,Z \in \G$.
\end{lem}
\begin{proof}
Fix all parameters and consider the sum $Y \oplus Z$ together with the canonical inclusion mappings $i_Y : Y \rightarrow Y \oplus Z$ and $i_Z : Z \rightarrow Y \oplus Z$. Equip $Y \oplus Z$ with a sequence $(\|\cdot\|_m)_{m < \omega}$ of seminorms defined by declaring $$\|(y,z)\|_m := \inf\{\|u\|_{Y,m} + \|v\|_{Z,m} + \varep \|x\|_{X,m} : \, x \in X, u \in Y, v \in Z,  y = u + f(x), z = v - g(x)\}$$ for each $m < n$, and $\|(y,z)\|_m := \|y\|_{Y,m} + \|z\|_{Z,m}$ for each $m \geq n$. Let $W$ be the sum $Y \oplus Z$ equipped with this sequence of seminorms. Note that $W$ is a graded multi-seminormed space whenever $X,Y$ and $Z$ are graded. As in the proof of Lemma \ref{NAP}, it is straightforward to check that the inclusion mappings $i_Y : Y \rightarrow Y \oplus Z$ and $i_Z : Z \rightarrow Y \oplus Z$ are multi-isometric embeddings which satisfy $\|i_Y \circ f - i_Z \circ g\|_m \leq \varep$ for all $m<n$.
\end{proof}

\begin{lem}
Let $X$ be a finite-dimensional subspace of $E$ such that $\la_X = \omega$. For every $\varep > 0, n \in \N$, finite-dimensional multi-seminormed space $Y$ and multi-isometric $n$-embeddings $\eta : X \rightarrow E$ and $\gamma : X \rightarrow Y$, there is a multi-isometric embedding $f : Y \rightarrow E$ such that $\| f \circ \gamma - \eta \|_m \leq \varep$ for each $m < n$.
\end{lem}
\begin{proof}
Apply the previous lemma to $\varep/2$, $\gamma : X \rightarrow Y$ and $\eta : X \rightarrow Z := \eta(X)$ to find the corresponding $W, i_Y$ and $i_Z$. Since $E$ is approximately ultrahomogeneous, there is $g \in \Iso(E)$ such that $\| g(i_Z(z)) - z \|_m \leq \frac{\varep}{2} \|z\|_m$ for all $z \in Z$ and $m \in \N$. Let $f = g\restriction_W \circ \, i_Y$. Then for each $x \in X$ and $m < n$ we have
\begin{equation*}
\begin{split}
\|f(\gamma(x)) - \eta(x)\|_m & \leq \|g(i_Y(\gamma(x))) - g(i_Z(\eta(x)))\|_m + \|g(i_Z(\eta(x))) - \eta(x)\|_m \\
& \leq \|g(i_Y(\gamma(x))-i_Z(\eta(x)))\|_m + \frac{\varep}{2}\|\eta(x)\|_m\\
& \leq \frac{\varep}{2}\|x\|_m + \frac{\varep}{2} \|x\|_m \leq \varep\|x\|_m.
\end{split}
\end{equation*}
Thus $f$ is the desired multi-isometric embedding.
\end{proof}

\begin{cor}
$(\GG^\omega, (\|\cdot\|_n)_{n < \omega})$ and $(\GG^\omega, (\|\cdot\|'_n)_{n < \omega})$ are approximately ultrahomogeneous for $\G_{<\omega}$ and $\M_{<\omega}$, respectively. In particular, $(\GG^\omega, (\|\cdot\|_n)_{n < \omega}) = \Flim(\G_{<\omega})$ and $(\GG^\omega, (\|\cdot\|'_n)_{n < \omega}) = \Flim(\M_{<\omega})$.
\end{cor}
\begin{proof}
We only prove the result for $\G_{<\omega}$; the proof for $\M_{<\omega}$ is identical. Fix $X \in \G_{<\omega}$ together with multi-isometric embeddings $\gamma, \eta : X \rightarrow E$. Extend the sequence of seminorms $(\|\cdot\|_{X, n})_{n < \la_X}$ to an infinite sequence in the natural way by declaring $\|\cdot\|_{X,m} = \|\cdot\|_{X,\la_X - 1}$ for all $m \geq \la_X$. Then $\gamma$ and $\eta$ become multi-isometric $\la_X$-embeddings. Let $Y = \gamma(X) \subseteq E$ be equipped with the sequence of seminorms from $E$ and apply the previous lemma to find a multi-isometric embedding $f : Y \rightarrow E$ such that $\|f(\gamma(x)) - \eta(x) \|_m \leq \frac{\varep}{2} \|x\|_m$ for each $x \in X$ and $m < \la_X$. By the approximate ultrahomogeneity of $E$, there is $g \in \Iso(E)$ such that $\|g(y) - f(y)\|_m \leq \frac{\varep}{2}\|y\|_m$ for each $y \in Y$ and $m \in \N$. Then $$\|g(\gamma(x)) - \eta(x)\|_m \leq \|g(\gamma(x)) - f(\gamma(x))\|_m + \|f(\gamma(x)) - \eta(x)\|_m \leq \varep \|x\|_m$$ for each $x \in X$ and $m < \la_X$, as required.
\end{proof}

\section{The approximate Ramsey property}

As in \cite{BLLM} or \cite{FLMT}, we can characterize the extreme amenability of the group of multi-isometries of certain Fr\'echet spaces in terms of an approximate Ramsey property. Before stating the relevant Ramsey properties, we need some terminology. Given $r \in \N$, an \emph{$r$-colouring} of a set $X$ is simply a mapping $X \rightarrow r$. If $X$ is equipped with a finite sequence of seminorms $(\|\cdot\|_m)_{m < \la_X}$ and $n \leq \la_X$, an \emph{$n$-continuous colouring} of $X$ is a mapping $c : X \rightarrow [0,1]$ such that $$|c(x) - c(y)| \leq \max_{m < n} \|x-y\|_m \text{ for all $x, y \in X$}.$$ Thus, an $n$-continuous colouring is simply a mapping of $X$ into $[0,1]$ which is 1-Lipschitz with respect to the pseudometric induced by $\max_{m < n} \|\cdot\|_m$.

\begin{defi}
Let $\K$ be a collection of finite-dimensional multi-seminormed spaces.
	\begin{enumerate}[(a)]
	\item $\K$ has the \emph{discrete approximate Ramsey property} (discrete ARP) if for every $X, Y \in \K, r \in \N$ and $\varep > 0$ there is $Z \in \K$ such that every $r$-colouring of $\Emb(X,Z)$ \emph{$\varep$-stabilizes} on a set of the form $\gamma \circ \Emb(X,Y)$ for some $\gamma \in \Emb(Y,Z)$, i.e. there is $i < r$ such that $\gamma \circ \Emb(X,Y)$ is contained in the set $$(c^{-1}\{i\})_\varep := \{\xi \in \Emb(X,Z) : \exists \eta \, \text{ $c(\eta) = i$ and $\|\xi-\eta\|_m \leq \varep$ for all $m < \la_X$}\},$$ where the $m^{\mathrm{th}}$ seminorm $\|\cdot\|_m$ on the space of embeddings $\ga : X \rightarrow Z$ is defined by setting $\|\ga\|_m := \sup \{ \|\ga(x)\|_m : x \in X, \|x\|_m =1 \}.$ In this case we will also say that $\gamma \circ \Emb(X,Y)$ is \emph{$\varep$-monochromatic}.
	\item $\K$ has the \emph{continuous approximate Ramsey property} (continuous ARP) if for every $X \in \K_{<\omega}, Y \in \K$ and $\varep > 0$ there is $Z \in \K$ such that for every $\la_X$-continuous colouring $c$ of $\Emb(X,Z)$ there is $\gamma \in \Emb(Y,Z)$ such that the \emph{oscillation} of $c$ on $\gamma \circ \Emb(X,Y)$, defined as $$\osc(c \restriction \gamma \circ \Emb(X,Y)) := \sup\{|c(\gamma \circ \eta) - c(\gamma \circ \eta')| : \eta, \eta' \in \Emb(X,Y)\},$$ is at most $\varep$. In this case we will say that $c$ \emph{$\varep$-stabilizes} on $\gamma \circ \Emb(X,Y)$.
	\end{enumerate}
\end{defi}

We will abbreviate the conclusions of the discrete and continuous ARP by writing $Z \rightarrow_\varep (Y)^X_r$ and $Z \rightarrow_\varep (Y)^X$, respectively. Exactly as in \cite{BLLM}, it turns out that these two notions are equivalent.

\begin{lem}
Let $\K \subseteq \M_{<\omega}$. Then $\K$ satisfies the discrete ARP if and only if it satisfies the continuous ARP.
\end{lem}
\begin{proof}  
Suppose first that $\K$ has the discrete ARP. Fix $X,Y \in \K$ and $\varep > 0$. Let $D \subseteq [0,1]$ be a finite $\varep$-dense set. Apply the discrete ARP with $|D|$-many colours to find $Z \in \K$ such that $Z \rightarrow_\varep (Y)^X_{|D|}$. We claim that $Z$ witnesses the continuous ARP for the above parameters. Indeed, given a $\la_X$-continuous colouring $c : \Emb(X,Z) \rightarrow [0,1]$, define a $|D|$-colouring $\widetilde{c} : \Emb(X,Z) \rightarrow D$ by the condition $|c(\varphi) - \widetilde{c}(\varphi)| \leq \varep$ for every $\varphi \in \Emb(X,Z)$. Then there is $\gamma \in \Emb(Y,Z)$ such that $\widetilde{c}$ $\varep$-stabilizes on $\gamma \circ \Emb(X,Y)$. It follows that $c$ $4\varep$-stabilizes on $\gamma \circ \Emb(X,Y)$.

Now suppose $\K$ has the continuous ARP. We prove that $\K$ has the discrete ARP by induction on $r$, the number of colours. When $r=1$ this is trivial, so suppose inductively that $\K$ satisfies the discrete ARP for $r$-colourings. Fix $X,Y \in \K$ and $\varep > 0$. By the inductive hypothesis, there is $Z_0 \in \K$ such that $Z_0 \rightarrow_\varep (Y)^X_r$. Since $\K$ has the continuous ARP, there is $Z \in \K$ such that $Z \rightarrow_\varep (Z_0)^X$. We claim that $Z$ witnesses the discrete ARP for the parameters $X,Y, \varep$ and $r+1$. Indeed, fix a colouring $c: \Emb(X,Z) \rightarrow r+1$ and define $\widetilde{c} : \Emb(X,Z) \rightarrow [0,1]$ by setting $$\widetilde{c}(\varphi) = \min\left\{1,  \inf_{\psi \in c^{-1}\{r\}} \max_{m < \la_X} \| \varphi - \psi \|_m  \right\}.$$ It is routine to check that $\widetilde{c}$ is an $\la_X$-continuous colouring, so there is $\gamma \in \Emb(Z_0, Z)$ such that $\widetilde{c}$ $\varep$-stabilizes on $\gamma \circ \Emb(X,Z_0)$. If there is some $\varphi \in \Emb(X,Z_0)$ such that $c(\gamma \circ \varphi) = r$, then $\gamma \circ \Emb(X, Z_0) \subseteq (c^{-1}\{r\})_\varep$ and so we are done since then $c$ $\varep$-stabilizes on $\gamma \circ \gamma_0 \circ \Emb(X,Y)$ for any choice of $\gamma_0 \in \Emb(Y,Z_0)$. If no such $\varphi$ exists, we can define an $r$-colouring $d$ of $\Emb(X,Z_0)$ by setting $d(\varphi) = c(\gamma \circ \varphi)$. By definition of $Z_0$, there is $\gamma_0 \in \Emb(Y,Z_0)$ such that $d$ $\varep$-stabilizes on $\gamma_0 \circ \Emb(X,Y)$. It then follows that $c$ $\varep$-stabilizes on $\gamma \circ \gamma_0 \circ \Emb(X,Y)$.
\end{proof}

The next lemma is a particular instance of a more general phenomenon which rephrases the Ramsey property of an age in terms of its limit. The proof is similar to that of \cite[Proposition 3.4]{MT}. Recall that $\Age_n(E)$ denotes the collection of all finite-dimensional linear subspaces $X \subseteq E$ equipped with the first $n$ seminorms from $E$.

\begin{lem}
Suppose $E$ is an approximately ultrahomogeneous multi-seminormed space. Then the collection $\Age_n(E)$ has the continuous ARP if and only if for every $X,Y \in \Age_n(E), \varep > 0$ and $\la_X$-continuous colouring $c$ of $\Emb(X,E)$, there is $\gamma \in \Emb(Y,E)$ such that $\osc(c\restriction \gamma \circ \Emb(X,Y)) \leq \varep$.
\end{lem}
\begin{proof}
The left-to-right direction is straightforward and will not be used in what follows, so we will only show the right-to-left direction. We will prove the contrapositive. First, let $H = \{\eta_1, \dots, \eta_k\}$ be a finite $\varep/3$-dense subset of $\Emb(X,Y)$, where the latter set is equipped with the pseudometric induced by $\max_{m < \la_X} \|\cdot\|_m$. We will need the following claim, the proof of which is routine.

\begin{claim}
Suppose there is $Z \in \Age_n(E)$ such that for every $\la_X$-continuous colouring of $\Emb(X,Z)$, there is $\gamma \in \Emb(Y,Z)$ such that $\osc(c \restriction \gamma \circ H) \leq \varep/3$. Then $Z \rightarrow_\varep (Y)^X$.
\end{claim}

Now, if $\Age_n(E)$ does not have the continuous ARP, then there are $X,Y \in \Age_n(E)$ and $\varep > 0$ such that no $Z \in \Age_n(E)$ witnesses $Z \rightarrow_\varep (Y)^X$. Thus, by the Claim, for each such $Z$ we can fix a bad $\la_X$-continuous colouring $c_Z$ such that $\osc(c_Z \restriction \gamma \circ H) \geq \varep$ for any choice of $\gamma \in \Emb(Y,Z)$. Fix an ultrafilter $\U$ on $\Age_n(E)$ such that $$\{W \in \Age_n(E) : V \subseteq W\} \in \U \text{ for each $V \in \Age_n(E)$}.$$ Define a mapping $c : \Emb(X, E) \rightarrow [0,1]$ by setting $c(\gamma) = \lim_\U c_Z(\gamma)$. Note that the ultralimit exists (by boundedness) and is well-defined since $\{W \in \Age_n(E) : \gamma(X) \subseteq W\} \in \U$. Furthermore, $c$ is an $\la_X$-continuous colouring. We claim that $c$ is a bad colouring of $\Emb(X,E)$. To this end, take any $\rho \in \Emb(Y,E)$ and note that $\{W \in \Age_n(E) : \rho(Y) \subseteq W\} \in \U$. Furthermore, for any such $W$ we have $\rho \in \Emb(Y,W)$ and so by our initial hypothesis we know $$| c_W(\rho \circ \eta_i) - c_W(\rho \circ \eta_j)| > \varep \text{ for some $i, j \in \{1, \dots, k\}$.}$$ Since $\U$ is an ultrafilter, it follows that there are $i, j \in \{1,\dots,k\}$ such that $$\{W \in \Age_n(E) : | c_W(\rho \circ \eta_i) - c_W(\rho \circ \eta_j)| > \varep\} \in \U.$$ It then follows that $|c(\rho \circ \eta_i) - c(\rho \circ \eta_j)| > \varep$. Since $\rho$ was arbitrary, we see that $c$ is a bad colouring of $\Emb(X,E)$.
\end{proof}

The following is the KPT correspondence for multi-seminormed spaces (cf. \cite{KPT, MT}).

\begin{thm}[Kechris-Pestov-Todor\v{c}evi\'c correspondence]\label{KPT}
Suppose $E$ is an infinite-dimensional multi-seminormed space which is approximately ultrahomogeneous. The following are equivalent:
	\begin{enumerate}[(i)]
	\item $\Age_n(E)$ has the ARP for each $n \in \N$ such that $1\leq n \leq \lambda_E$.
	\item $\Iso(E)$ is extremely amenable when endowed with the topology of pointwise convergence.
	\end{enumerate}
\end{thm}
\begin{proof}
$(i) \rightarrow (ii)$: Fix an $\Iso(E)$-flow $\Iso(E) \acts K$ for $K$ compact, $\varep > 0, p \in K$, an entourage $U$ and a finite $F \subseteq \Iso(E)$. By one of the well-known characterizations of extreme amenability (see \cite{Pestov} or \cite[Claim 5.11.2]{FLMT}) it is enough to find $g \in \Iso(E)$ such that $F \cdot (g \cdot p)$ is $U$-small, i.e. such that $$(f_0 \cdot (g \cdot p), f_1 \cdot (g \cdot p)) \in U \text{ for each $f_0, f_1 \in F$}.$$ Before proceeding, we will define a directed family of pseudometrics which generate the left uniformity of $\Iso(E)$: For each $n \in \N$ such that $n \leq \lambda_E$ and each finite-dimensional $X \subseteq E$, define a pseudometric $d_X^n$ on $\Iso(E)$ by $$d_X^n(g,h) = \max_{m<  n} \|g\restriction X -h\restriction X\|_m.$$ We can assume without loss of generality that all entourages are symmetric. Fix an entourage $V$ such that $V \circ V \circ V \circ V \subseteq U$. Since the mapping $\Iso(E) \rightarrow K : g \rightarrow g^{-1} \cdot p$ is left uniformly continuous (see, e.g., \cite[Lemma 2.1.5]{Pestov}), there are $n$, $X \subseteq E$ and $\delta > 0$ such that $$d_X^n(g,h) \leq \delta \text { implies } (g^{-1} \cdot p, h^{-1} \cdot p) \in V.$$

Equip $X$ with the first $n$ seminorms induced from $E$, so that $X \in \Age_n(E)$. Let $Y$ be any member of $\Age_n(E)$ containing $\bigcup \{g(X) : g^{-1} \in F \cup \{\id\} \}$. Fix a finite set $\{x_i\}_{i<r} \subseteq K$ such that $K \subseteq \bigcup_{i<r} V[x_i]$, where $V[x] = \{y \in K : (x,y) \in V\}$. Apply the approximate Ramsey property of $\Age(E)$ to the parameters $n, X,Y, \delta/3$ and $r$ to obtain $Z \in \Age_n(E)$ such that $$Z \rightarrow_{\delta/3, n} (Y)^X_r.$$ Define a colouring $c : \Emb(X,Z) \rightarrow r$ by first choosing, for each $\gamma \in \Emb(X,Z)$, some $g_\gamma \in \Iso(E)$ such that $\max_{m < n} \|g_\gamma \restriction X - \gamma\|_m \leq \delta/3$; such a choice is possible by the approximate ultrahomogeneity of $E$. Then define $c(\gamma) = i$ if $i < r$ is the least index such that $g_\gamma^{-1} \cdot p \in V[x_i]$. By definition of $Z$ there are $\rho \in \Emb(Y,Z)$ and $i < r$ such that $$\rho \circ \Emb(X,Y) \subseteq (c^{-1}\{i\})_{\delta/3, n}.$$ In particular, this implies that for each $\eta \in \Emb(X,Y)$ there is $h_\eta \in \Iso(E)$ such that $\max_{m < n} \|\rho \circ \eta - h_\eta\restriction X\|_m \leq 2\delta/3$ and $h_\eta^{-1} \cdot p \in V[x_i]$. Choose $g \in \Iso(E)$ such that $\max_{m < n} \|g \restriction Y - \rho\|_m \leq \delta/3$. Now, given $f_0, f_1 \in F$, let $\eta_j := f_j^{-1} \restriction X$ for $j = 0, 1$ and note $\eta_j \in \Emb(X,Y)$. Then for each $j = 0, 1$,
\begin{equation*}
\begin{split}
d_X^n(g\circ f_j^{-1}, h_{\eta_j}) & = \max_{m < n} \|g \circ \eta_j - h_{\eta_j}\restriction X\|_m \\
& \leq \max_{m < n} \|g \circ \eta_j - \rho \circ \eta_j\|_m + \max_{m < n} \|\rho \circ \eta_j - h_{\eta_j} \restriction X\|_m \\
& \leq \delta/3 + 2\delta/3 = \delta.
\end{split}
\end{equation*}
By choice of $\delta$, this implies $(f_j \circ g^{-1} \cdot p, h_{\eta_j}^{-1} \cdot p) \in V$ for each $j$. Since $(x_i, h_{\eta_j}^{-1} \cdot p) \in V$ for each $j$, our choice of $V$ implies $(f_0 \circ g^{-1} \cdot p, f_1 \circ g^{-1} \cdot p) \in U$. Thus $F \cdot (g^{-1} \cdot p)$ is $U$-small.

\medskip \noindent $(ii) \rightarrow (i)$: To prove the ARP, we use the previous two lemmata together with the following characterization of extreme amenability in terms of the family of pseudometrics defined above. (See \cite{Pestov} or \cite[Proposition 3.9]{MT}.)
\begin{enumerate}
	\item[$(\ast)$] $\Iso(E)$ is extremely amenable if, and only if, for every finite $F \subseteq \Iso(E), \varep > 0, X \subseteq E, n \in [1,\lambda_E] \cap \N$ and $1$-Lipschitz map $f : (\Iso(E), d_X^n) \rightarrow [0,1]$, there is $g \in \Iso(E)$ such that $\osc(f \restriction g \circ F) \leq \varep$.
\end{enumerate}
 Fix $X, Y \in \Age_n(E)$ together with $\varep > 0$ and an $n$-continuous colouring $c$ of $\Emb(X,E)$. Let $H = \{\eta_1, \dots, \eta_k\}$ be a finite $\varep$-dense subset of $\Emb(X,Y)$ where the latter set is equipped with the pseudometric induced by $\max_{m < n} \|\cdot\|_m$. Apply the approximate ultrahomogeneity of $E$ to find $F = \{g_1, \dots, g_k\} \subseteq \Iso(E)$ such that $\|g_i \restriction X - \eta_i\|_m \leq \varep$ for all $m < n$ and all $i \leq k$. Let $\widetilde{c} : (\Iso(E), d_n^X) \rightarrow [0,1]$ be the 1-Lipschitz mapping defined by $\widetilde{c}(g) = c(g\restriction X)$ and use $(\ast)$ to find $g \in \Iso(E)$ such that $\osc(\widetilde{c} \restriction g \cdot F) \leq \varep$. Then for any $i, j \leq k$, by the triangle inequality the term $|c(g \circ \eta_i) - c(g \circ \eta_j)|$ is bounded above by $$|c(g \circ \eta_i) - c(g \circ g_i\restriction X)| + |c(g \circ g_i \restriction X) - c(g \circ g_j \restriction X)| + |c(g \circ g_j \restriction X) - c(g \circ \eta_j)|.$$ Since $c$ is an $n$-continuous colouring, the first term is bounded above by $$\max_{m < n} \|g \circ \eta_i - g \circ g_i\restriction X\|_m = \max_{m < n} \|\eta_i - g_i\restriction X\|_m \leq \varep$$ by our choice of $g_i$. Similarly, the third term is bounded above by $\varep$. To bound the second term, note that $$|c(g \circ g_i \restriction X) - c(g \circ g_j \restriction X)| = |\widetilde{c}(g \circ g_i) - \widetilde{c}(g \circ g_j)| \leq \varep$$ by our choice of $g$. Thus, if we let $\gamma = g \restriction Y$, we see that the oscillation of $c$ on $\gamma \circ H$ is bounded by $3\varep$. Then, using the definition of the $\eta_i$, it follows that $\osc(\gamma \circ \Emb(X,Y)) \leq 5\varep$.
\end{proof}

Our next goal is to show that various classes of finite-dimensional multi-seminormed spaces have the ARP, which will allow us to apply the KPT correspondence in certain instances to obtain examples of extremely amenable multi-isometry groups. As in the case of the amalgamation property, our main source of examples of such classes will come from classes of finite-dimensional normed spaces which are known to have the ARP in the context of normed spaces. The key lemma is the following:

\begin{lem}\label{ARP1}
Suppose $\K_1, \dots, \K_n$ are classes of finite-dimensional normed spaces with the ARP. Let $X_1,\dots, X_n, Y_1,\dots, Y_n$ be finite-dimensional seminormed spaces such that $(X_i)_{\|\cdot\|}$ and $(Y_i)_{\|\cdot\|}$ belong to $\K_i$ for each $i \leq n$. Let $\varep>0$ and $r < \omega$. Then there are finite-dimensional seminormed spaces  $Z_1,\dots, Z_n$ such $(Z_i)_{\|\cdot\|} \in \K_i$ for each $i$ and, for every colouring $c: \prod_{j=1}^n \Emb(X_j, Z_j)\to r$, there are $\rho_j\in \Emb(Y_j,Z_j)$, $j=1,\dots,n$, such that $$\prod_{j=1}^n \rho_j \circ \Emb(X_j,Y_j) \text{ is $\varep$-monochromatic}.$$    
\end{lem}

The proof will involve a standard strategy for obtaining product Ramsey properties. First we will need:

\begin{lem}\label{basecase}
Suppose $\K$ is a class of finite-dimensional normed spaces with the ARP. For any finite-dimensional seminormed spaces $X$ and $Y$ such that $X_{\|\cdot\|}$ and $Y_{\|\cdot\|}$ belong to $\K$, any $\varep>0$ and $r\in \N$, there is a finite-dimensional seminormed space $Z$ such that $Z_{\|\cdot\|} \in \K$ and every colouring $c: \Emb(X, Z) \rightarrow r$ $\varep$-stabilizes on a set of the form $\rho \circ \Emb(X,Y)$ for $\rho \in \Emb(Y,Z)$.
\end{lem}
\begin{proof}
Let $X, Y$ be finite-dimensional seminormed spaces, $\varep > 0$ and $r \in \N$. Let $\widetilde{X} = X / \ker \| \cdot \|_X$ and $\widetilde{Y} = Y / \ker \|\cdot \|_Y$ be equipped with the norms $\|[x]\| = \|x\|_X$ and $\|[y]\| = \|y\|_Y$ respectively. By the ARP of $\K$, there is a finite-dimensional normed space $Z \in \K$ such that $$Z \rightarrow_\varep (\widetilde{Y})^{\widetilde{X}}_r.$$ Consider the product $Z \times Z$ equipped with the seminorm $\|(z_1,z_2)\| := \|z_1\|_Z$. We claim that this space witnesses the ARP for the parameters $X, Y, \varep, r$. Note that $Z \times Z \in \K_{\|\cdot\|}$ since the associated quotient is isomorphic to $Z$. Now, fix a colouring $c : \Emb(X,Z) \rightarrow r$ and define $\widetilde{c} : \Emb(\widetilde{X},Z) \rightarrow r$ by $\widetilde{c}(\gamma) = c(\gamma \circ \pi_X)$ where $\pi_X : X \rightarrow \widetilde{X}$ is the canonical surjection. By definition of $Z$, there is $\rho \in \Emb(\widetilde{Y}, Z)$ such that $$\rho \circ \Emb(\widetilde{X}, \widetilde{Y}) \subseteq (\widetilde{c}^{-1}\{i\})_\varep \, \text{ for some $i < r$}.$$ Let $\bar{\rho} : \widetilde{Y} \rightarrow Z \times Z$ be defined by setting $\bar{\rho}(y) = (\rho(y), 0)$. We will show $$(\bar{\rho} \circ \pi_Y) \circ \Emb(X,Y) \subseteq (c^{-1}\{i\})_\varep.$$ To this end, fix $\eta \in \Emb(X,Y)$ and define a mapping $\varphi : \widetilde{X} \rightarrow \widetilde{Y}$ by $\varphi([x]) = \pi_Y(\eta(x))$; it is easy to check that $\varphi$ is a well-defined multi-isometric embedding which satisfies $\varphi \circ \pi_X = \pi_Y \circ \eta$. Then by definition of $\rho$ there is $\theta \in \widetilde{c}^{-1}\{i\}$ such that $\| \rho \circ \varphi - \theta \| \leq \varep.$ Let $\bar{\theta} : \widetilde{X} \rightarrow Z \times Z$ be given by $\bar{\theta}(x) = (\theta(x), 0)$. The situation is summarized by the following diagram:
\begin{center}
\begin{tikzcd}[row sep=large, column sep=large]
X \arrow[r, "\pi_X"] \arrow[d, "\eta"'] & \widetilde{X} \arrow[d, "\varphi"] \arrow[rd, "\bar{\theta}"{name=M}] &   \\
Y \arrow[ur, phantom, "\scalebox{1.4}{$\circlearrowleft$}" description]\arrow[r, "\pi_Y"']                    & \widetilde{Y} \arrow[r, "\bar{\rho}"']    \arrow[ur, phantom, "\scalebox{0.8}{$\cca{\varep}$}" description, to=M]                   & Z \times Z
\end{tikzcd}
\end{center}
Then $c(\theta \circ \pi_X) = i$ and, since $$(\rho \circ \pi_Y) \circ \eta = \rho \circ (\pi_Y \circ \eta) = \rho \circ (\varphi \circ \pi_X),$$ we have $$\|(\rho \circ \pi_Y) \circ \eta - \theta \circ \pi_X \| = \|(\rho \circ \varphi) \circ \pi_X - \theta \circ \pi_X\| \leq \|\rho \circ \varphi - \theta \| \leq \varep.$$ Thus $\rho \circ \pi_Y$ is the desired embedding.
\end{proof}
  
\begin{proof}[Proof of Lemma \ref{ARP1}]
The proof is by induction on $n$. The case $n=1$ follows from the previous lemma, so fix all parameters and suppose that we can find seminormed spaces $Z_1,\dots,Z_n$ such that every $r$-colouring of $\prod_{j=1}^n \Emb(X_j, Z_j)$ has a $\varep/2$-monochromatic set of the form $\prod_{j=1}^n \rho_j\circ \Emb(X_j, Y_j)$. Let $D$ be a finite $\varep/2$-dense subset of  $\prod_{j=1}^n \Emb(X_j,Z_j)$, where the seminorm on the product is the maximum of the given seminorms.  Apply Lemma \ref{basecase} to $X_{n+1}$ and $Y_{n+1}$ to find a seminormed space $Z$ that works for the error $\varep/2$, and the number of colours being the cardinality of $^{D}{r}$. We claim that $Z_1,\dots,Z_n,Z$ works.  For suppose that $c:\prod_{j=1}^{n+1} \Emb(X_j,Z_j)\to r$. We have the induced colouring $\widehat c: \Emb(X_{n+1},Z)\to \, ^{D}r$,  $$\widehat c(\xi)(\eta_1,\dots,\eta_n):=c(\eta_1,\dots,\eta_n, \xi)\in r \text{ for every $(\eta_1,\dots,\eta_n)\in D$}.$$ By the choice of $Z$ there is some $\rho\in \Emb(Y_{n+1},Z)$ such that $\rho \circ \Emb(X_{n+1},Y_{n+1})$ is $\varep/2$-monochromatic for $\widehat c$ with colour $\theta\in {}^{D}r$. The mapping $\theta: D\to r$ defines an $r$-colouring $\widehat{\theta}:\prod_{j=1}^n \Emb(X_j, Z_j)\to r$ by $\varep/2$-proximity. For each $j=1,\dots, n$, let $\rho_j\in \Emb(Y_j, Z_j)$ be such that $\prod_{j=1}^n \rho_j \circ \Emb(X_j, Y_j)$ is $\varep/2$-monochromatic for $\widehat\theta$ with colour $s\in r$. Then the set $$\left(\prod_{j=1}^n \rho_j \circ \Emb(X_j, Y_j) \right)\times (\rho \circ \Emb(X_{n+1},Y_{n+1})) \subseteq \prod_{j=1}^{n+1} \Emb(X_j,Z_j)$$ is $\varep$-monochromatic  for $c$ with colour $s$:    Let $(\gamma_j)_{j=1}^{n+1}\in \prod_{j=1}^{n+1}\Emb(X_j, Y_j)$.  There is some $(\mu_{j})_{j=1}^n \in \prod_{j=1}^n \Emb(X_j, Z_j)$ such that $\widehat \theta((\mu_j)_{j=1}^n)=s$ and $(\mu_j)_{j=1}^n$ is $\varep/2$-close to $(\rho_j \circ \gamma_j)_{j=1}^n$. Choose $(\pi_j)_{j=1}^n\in D$ that is $\varep/2$-close to  $(\mu_j)_{j=1}^n$  such that $\theta((\pi_j)_{j=1}^n)=s$.   Let $\mu\in \Emb(X,Z)$ be such that $\mu$ is $\varep/2$-close to $\rho \circ \gamma$ and $\widehat c(\mu)=\theta$.  This last equality means by definition that $$s=\theta((\mu_j)_{j=1}^n)= \widehat c(\mu)((\mu_j)_{j=1}^n)=c(\mu_1,\dots,\mu_n,\mu).$$ Finally, observe that $\mu$ is $\varep/2$-close to $\rho$, and each $\mu_j$ is $\varep$-close to $\rho_j \circ \gamma_j$. 
\end{proof}

For the next proposition, recall the definition of the classes $\langle \bar{\K}\rangle_{=n}$ from Definition \ref{classes}.

\begin{prop}\label{ARP}
Let $\bar{\K} = \{\K_n\}_{n<\omega}$ be a collection of finite-dimensional normed spaces such that each $\K_n$ has the ARP. For every $n \in \N, X,Y\in \langle \bar{\K}\rangle_{=n}$, $r\in \N$, and every $\varep>0$ there is $Z\in \langle \bar{\K}\rangle$ such that every $r$-colouring of $\Emb(X,Z)$ has an $\varep$-monochromatic set of the form $\rho \circ \Emb(X,Y)$ for some $\rho\in \Emb(Y,Z)$. Furthermore, if $\K_n = \K$ for each $n$ and $\K$ is closed under $\ell_\infty$-sums, then $Z$ can be chosen to be graded when $X$ and $Y$ are graded.
\end{prop}
\begin{proof}
Fix all parameters and apply the previous lemma to $(X,\nrm{\cdot}_j)_{j=1}^n$,  $(Y,\nrm{\cdot}_j)_{j=1}^n$, $\varep$ and $r$ to find the corresponding $Z_1,\dots, Z_n$. Let $Z:= \prod_{j=1}^n Z_j$, and for each $j=1,\dots,n$, let $$\nrm{(z_1,\dots,z_n)}_j:= \nrm{z_j}_{Z_j}.$$ Note that $Z \in \langle \bar{\K}\rangle_{=n}$ by construction. We claim that $Z$ witnesses the ARP for the given parameters. For suppose that $c: \Emb(X,Z)\to r$ is a colouring. Given a sequence $\vec{\gamma} = (\gamma_j)_j \in \prod_{j=1}^n \Emb((X,\nrm{\cdot}_j), Z_j)$, define a mapping $F(\vec{\gamma}) : X \rightarrow Z$ by $$F(\vec{\gamma})(x) = (\gamma_1(x), \dots, \gamma_n(x)).$$ Observe that $F(\vec{\gamma}) \in \Emb(X,Z)$ since each $\gamma_j$ is a multi-isometric embedding. Using this mapping, define an induced colouring $\widehat c: \prod_{j=1}^n \Emb((X,\nrm{\cdot}_j), Z_j)\to r$, by $\widehat{c}(\vec{\gamma}) = c(F(\vec{\gamma}))$ where $\vec{\gamma} = (\gamma_j)_j$. By definition of each $Z_j$, there are $\rho_j \in \Emb((Y,\|\cdot\|_j), Z_j)$ such that  $$\prod_{j=1}^n \rho_j \circ \Emb((X,\|\cdot\|_j),(Y_j,\|\cdot\|_j)) \text{ is $\varep$-monochromatic}.$$ Define $\rho \in \Emb(Y,Z)$ by $\rho(y) = (\rho_1(y), \dots, \rho_n(y))$. Note that $\rho \circ \eta = F((\rho_j \circ \eta)_j)$ and so $c(\rho \circ \eta) = \widehat{c}((\rho_j \circ \eta)_j)$. In particular, it follows that $\rho \circ \Emb(X,Y)$ is $\varep$-monochromatic.

In the case where $\K_n = \K$ for all $n$ and $\K$ is closed under $\ell_\infty$-sums, we work with the same underlying space $Z$ but we instead equip it with the sequence of seminorms given by $$\nrm{(z_1,\dots,z_n)}_j:= \max_{i \leq j} \nrm{z_i}_{Z_i}.$$ Then $Z_{\|\cdot\|_j}$ is multi-isometric to the $\ell_\infty$-sum of $Z_i, i \leq j$, and so $Z \in \langle \bar{\K}\rangle_{=n}$ by our additional assumption on $\K$. The rest of the proof is identical to that of the general case.
\end{proof}  

By appealing to the various known approximate Ramsey properties of classes of finite-dimensional normed spaces as considered in \cite{BLLM, FLMT}, we can apply Theorem \ref{KPT} and Proposition \ref{ARP} together to obtain the following result. The notation used below corresponds to that of Example \ref{examples}.

\begin{thm}
The following groups are extremely amenable when equipped with the topology of pointwise convergence:
	\begin{enumerate}[(1)]
		\item $\Iso(\GG^\omega, (\|\cdot\|_n)_{n < \alpha})$ and $\Iso(\GG^\omega, (\|\cdot\|'_n)_{n <\alpha})$ for each $\alpha \leq \omega$. In particular, the multi-isometry group of the separable (graded) Fr\'echet space of almost universal disposition for the class $\M_\omega$ (resp. $\G_\omega$) is extremely amenable.
		\item $\Iso(\Flim(\M_{<\omega}^{\mathcal H}))$.
		\item $\Iso(\Flim(\M_{<\omega}^{(p_n)}))$ for any sequence $(p_n) \subseteq [1,\infty[$ with $p_n \not \in \{4,6,8,\dots\}$.
	\end{enumerate} 
\end{thm}

The extreme amenability of the groups $\Iso(\GG^\omega, (\|\cdot\|_n)_{n < \omega})$ and $\Iso(\GG^\omega, (\|\cdot\|'_n)_{n <\omega})$ should naturally be compared to the extreme amenability of $\Iso(\GG)$. The following important questions remain open:

\begin{que}
Are $\Iso(\GG^\omega, (\|\cdot\|_n)_{n < \omega})$ and $\Iso(\GG^\omega, (\|\cdot\|'_n)_{n <\omega})$ topologically distinguishable from $\Iso(\GG)$? Are $\Iso(\GG^\omega, (\|\cdot\|_n)_{n < \omega})$ and $\Iso(\GG^\omega, (\|\cdot\|'_n)_{n <\omega})$ universal Polish groups?
\end{que}

\section{Concluding remarks}

Throughout this paper we have adopted a very particular viewpoint in order to develop Fra\"iss\'e theory for Fr\'echet spaces. Specifically, we have been working with the category of multi-seminormed spaces with morphisms given by seminorm-preserving linear mappings. A natural question is whether or not a similar theory can be developed for more relaxed categories which capture the notion of a Fr\'echet space. For instance, the following general problem remains open:

\begin{prob}
Develop a theory of \emph{metric} Fr\'echet spaces, i.e. pairs of the form $(X,d)$ where $X$ is a vector space and $d$ is a translation-invariant metric inducing a Fr\'echet topology on $X$.
\end{prob}

In the above setting, one may be able to use more general machinery in order to develop a Fra\"iss\'e theory for such spaces, e.g. as in \cite{BY,Kubis}. It is unclear if the use of such machinery is possible in the setting of multi-seminormed spaces, since in general one needs to work with an arbitrarily large (finite) number of seminorms.

One specific corollary of our construction is that the spaces $\Flim(\M_{<\omega})$ and $\Flim(\G_{<\omega})$ are -- in addition to their defining properties -- Fra\"iss\'e for the classes $\M_\omega$ and $\G_\omega$, respectively, by virtue of being multi-isometric to the spaces constructed in \cite{BKK}. Since the construction of the Fra\"iss\'e limit makes use of the fact that the elements of a given Fra\"iss\'e class are finitely-seminormed, it is not clear that the aforementioned property must be true for general Fra\"iss\'e Fr\'echet spaces. More precisely, given a class $\K \subseteq \M_{<\omega}$, one can define a class $\K_\omega \subseteq \M_\omega$ by declaring that $(X,(\|\cdot\|_n)_{n<\omega}) \in \K_\omega$ exactly when 
$$(X,(\|\cdot\|_n)_{n<m}) \in \K \text{ for all $m<\omega$}.$$ This motivates the following:

\begin{prob}\label{prob2}
If $E$ is a $\K$-Fra\"iss\'e Fr\'echet space, is $E$ necessarily $\K_\omega$-Fra\"iss\'e?
\end{prob}

The proof that $\GG^\omega$ is $\G_\omega$-Fra\"iss\'e (with an appropriate sequence of seminorms) from \cite{BKK} makes use of the existence of a universal operator $\pi : \GG \rightarrow \GG$ constructed in \cite{GK}. This operator can essentially be seen as a \emph{Fra\"iss\'e} operator in a precise sense (see, e.g., \cite{Lupini2018}). The operator constructed in \cite{GK} is the Gurarij analogue of Rota's universal operator on $\ell_2$, as in \cite{Caradus, Rota}. Thus, in order to shed light on a possible affirmative answer to Problem \ref{prob2}, one may need to construct similar operators for an arbitrary Fra\"iss\'e Banach space in place of $\GG$.

Finally, we mention that throughout the paper we have not assumed any condition on the continuity of the relevant mappings between multi-seminormed spaces. Since seminorm-preserving mappings (as defined above) need not be continuous, the following general problem presents itself:

\begin{prob}\label{prob3}
For various classes $\K$ of finite-dimensional multi-seminormed spaces, study the notions of $\K$-universality and $\K$-Fra\"iss\'e where the embeddings are in addition assumed to be continuous. In particular, is there a version of the Fra\"iss\'e theory for multi-seminormed spaces developed above in which the associated embeddings are also required to be continuous?
\end{prob}

\nocite{BLLM_2}
\bibliographystyle{abbrv}
\bibliography{bibliography}

\begin{thebibliography}{10}

\bibitem{Banach}
S.~Banach.
\newblock {\em Th\'{e}orie des op\'{e}rations lin\'{e}aires}.
\newblock Chelsea Publishing Co., New York, 1955.

\bibitem{BKK}
C.~Bargetz, J.~K\k{a}kol, and W.~Kubi\'{s}.
\newblock A separable {F}r\'{e}chet space of almost universal disposition.
\newblock {\em J. Funct. Anal.}, 272(5):1876--1891, 2017.

\bibitem{BLLM_2}
D.~Barto\v{s}ov\'{a}, J.~Lopez-Abad, M.~Lupini, and B.~Mbombo.
\newblock The {R}amsey properties for {O}perator spaces and noncommutative
  {C}hoquet simplices.
\newblock Preprint, arXiv:2006.04799, 2020.

\bibitem{BLLM}
D.~Barto\v{s}ov\'{a}, J.~Lopez-Abad, M.~Lupini, and B.~Mbombo.
\newblock The {R}amsey property for {B}anach spaces and {C}hoquet simplices.
\newblock {\em J. Eur. Math. Soc.}, to appear.

\bibitem{BY1}
I.~Ben~Yaacov.
\newblock The linear isometry group of the {G}urarij space is universal.
\newblock {\em Proc. Amer. Math. Soc.}, 142(7):2459--2467, 2014.

\bibitem{BY}
I.~Ben~Yaacov.
\newblock Fra\"{\i}ss\'{e} limits of metric structures.
\newblock {\em J. Symb. Log.}, 80(1):100--115, 2015.

\bibitem{CFR}
F.~Cabello~S\'{a}nchez, V.~Ferenczi, and B.~Randrianantoanina.
\newblock On {M}azur rotations problem and its multidimensional versions.
\newblock {\em Preprint}, arXiv:2012.08344.

\bibitem{Caradus}
S.~R. Caradus.
\newblock Universal operators and invariant subspaces.
\newblock {\em Proc. Amer. Math. Soc.}, 23:526--527, 1969.

\bibitem{Carothers}
N.~L. Carothers.
\newblock {\em A short course on {B}anach space theory}, volume~64 of {\em
  London Mathematical Society Student Texts}.
\newblock Cambridge University Press, Cambridge, 2005.

\bibitem{FLMT}
V.~Ferenczi, J.~Lopez-Abad, B.~Mbombo, and S.~Todorcevic.
\newblock Amalgamation and {R}amsey properties of {$L_ p$} spaces.
\newblock {\em Adv. Math.}, 369:107190, 76, 2020.

\bibitem{Fraisse}
R.~Fra\"{\i}ss\'{e}.
\newblock Sur l'extension aux relations de quelques propri\'{e}t\'{e}s des
  ordres.
\newblock {\em Ann. Sci. Ecole Norm. Sup. (3)}, 71:363--388, 1954.

\bibitem{GK}
J.~Garbuli\'{n}ska-W\c{e}grzyn and W.~Kubi\'{s}.
\newblock A universal operator on the {G}urari\u{\i} space.
\newblock {\em J. Operator Theory}, 73(1):143--158, 2015.

\bibitem{GR}
R.~L. Graham and B.~L. Rothschild.
\newblock Ramsey's theorem for {$n$}-parameter sets.
\newblock {\em Trans. Amer. Math. Soc.}, 159:257--292, 1971.

\bibitem{Gurarij}
V.~I. Gurari\u{\i}.
\newblock Spaces of universal placement, isotropic spaces and a problem of
  {M}azur on rotations of {B}anach spaces.
\newblock {\em Sibirsk. Mat. \v{Z}.}, 7:1002--1013, 1966.

\bibitem{Hamilton}
R.~S. Hamilton.
\newblock The inverse function theorem of {N}ash and {M}oser.
\newblock {\em Bull. Amer. Math. Soc. (N.S.)}, 7(1):65--222, 1982.

\bibitem{Hodges}
W.~Hodges.
\newblock {\em A shorter model theory}.
\newblock Cambridge University Press, Cambridge, 1997.

\bibitem{K77}
N.~J. Kalton.
\newblock Universal spaces and universal bases in metric linear spaces.
\newblock {\em Studia Math.}, 61(2):161--191, 1977.

\bibitem{KPT}
A.~S. Kechris, V.~G. Pestov, and S.~Todorcevic.
\newblock Fra\"{\i}ss\'{e} limits, {R}amsey theory, and topological dynamics of
  automorphism groups.
\newblock {\em Geom. Funct. Anal.}, 15(1):106--189, 2005.

\bibitem{Kubis}
W.~Kubi\'{s}.
\newblock Metric-enriched categories and approximate {F}ra\"ss\'e limits.
\newblock {\em Preprint}, arXiv:1210.6506.

\bibitem{KS}
W.~Kubi\'{s} and S.~Solecki.
\newblock A proof of uniqueness of the {G}urari\u{\i} space.
\newblock {\em Israel J. Math.}, 195(1):449--456, 2013.

\bibitem{Kutateladze}
S.~S. Kutateladze.
\newblock {\em Fundamentals of functional analysis}, volume~12 of {\em Kluwer
  Texts in the Mathematical Sciences}.
\newblock Kluwer Academic Publishers Group, Dordrecht, 1996.
\newblock Translated from the second (1995) edition.

\bibitem{Lupini2018}
M.~Lupini.
\newblock Fra\"{\i}ss\'{e} limits in functional analysis.
\newblock {\em Adv. Math.}, 338:93--174, 2018.

\bibitem{Lusky}
W.~Lusky.
\newblock The {G}urarij spaces are unique.
\newblock {\em Arch. Math. (Basel)}, 27(6):627--635, 1976.

\bibitem{MV}
R.~Meise and D.~Vogt.
\newblock {\em Introduction to functional analysis}, volume~2 of {\em Oxford
  Graduate Texts in Mathematics}.
\newblock The Clarendon Press, Oxford University Press, New York, 1997.
\newblock Translated from the German by M. S. Ramanujan and revised by the
  authors.

\bibitem{MT}
J.~Melleray and T.~Tsankov.
\newblock Extremely amenable groups via continuous logic.
\newblock {\em Preprint}, arXiv:1708.01317.

\bibitem{Pestov}
V.~Pestov.
\newblock {\em Dynamics of infinite-dimensional groups}, volume~40 of {\em
  University Lecture Series}.
\newblock American Mathematical Society, Providence, RI, 2006.
\newblock The Ramsey-Dvoretzky-Milman phenomenon, Revised edition of {{\i}t
  Dynamics of infinite-dimensional groups and Ramsey-type phenomena} [Inst.
  Mat. Pura. Apl. (IMPA), Rio de Janeiro, 2005; MR2164572].

\bibitem{Rolewicz}
S.~Rolewicz.
\newblock {\em Metric linear spaces}.
\newblock PWN-Polish Scientific Publishers, Warsaw, 1972.
\newblock Monografie Matematyczne, Tom. 56. [Mathematical Monographs, Vol. 56].

\bibitem{Rota}
G.-C. Rota.
\newblock On models for linear operators.
\newblock {\em Comm. Pure Appl. Math.}, 13:469--472, 1960.

\bibitem{Semadeni}
Z.~Semadeni.
\newblock {\em Banach spaces of continuous functions. {V}ol. {I}}.
\newblock PWN---Polish Scientific Publishers, Warsaw, 1971.
\newblock Monografie Matematyczne, Tom 55.

\bibitem{Simon}
J.~Simon.
\newblock {\em Banach, {F}r\'{e}chet, {H}ilbert and {N}eumann spaces}.
\newblock Mathematics and Statistics Series. ISTE, London; John Wiley \& Sons,
  Inc., Hoboken, NJ, 2017.
\newblock Analysis for PDEs set. Vol. 1.

\bibitem{vogt1987}
D.~Vogt.
\newblock Operators between {F}r{\'e}chet spaces.
\newblock Analysis Conference Manila, 1987.

\end{thebibliography}

\end{document}